\documentclass[11pt]{amsart} 
\usepackage{verbatim, latexsym, amssymb, amsmath,color,mathrsfs}
\usepackage[linktoc=all]{hyperref}
\hypersetup{
	colorlinks,
	citecolor=black,
	filecolor=black,
	linkcolor=black,
	urlcolor=black
}

\def\R{\mathbb R}

\def\N{\mathbb N}
\def\Z{\mathbb Z}

\def\area{\mathrm{area}}

\newtheorem*{conj}{Conjecture}
\newtheorem*{prob}{Problem}

\newtheorem{thm}{Theorem}[section]
\newtheorem{lemm}[thm]{Lemma}
\newtheorem{cor}[thm]{Corollary}
\newtheorem{prop}[thm]{Proposition}
\theoremstyle{remark}

\theoremstyle{definition}

\title[]{Counting Minimal Lagrangians via Mirzakhani Functions}
\author{Ben Lowe, Fernando C. Marques, and Andr\'e Neves}
\address{University of Chicago \\ Department of Mathematics \\ Chicago IL 60637\\ USA}
\email{{loweb24@gmail.com}}
\address{Princeton University \\ Fine Hall \\ Princeton NJ 08544 \\ USA}
\email{coda@math.princeton.edu}
\address{University of Chicago \\ Department of Mathematics \\ Chicago IL 60637\\ USA}
\email{aneves@math.uchicago.edu}
\thanks{ The second  author is partly supported by NSF-DMS-2506810 and a Simons Investigator Grant. The third author is partly supported by NSF-DMS-2005468 and a Simons Investigator Grant.}

\begin{document}
	
% Rafi-Souto instead of Souto-Rafi.

	\maketitle
	%\ben{I added hyperlinks (change to the preamble)}
	\begin{abstract}
		We show that {for $k>1$} the number of genus $k$ minimal Lagrangians  with area at most $A$ in a product of hyperbolic surfaces grows on the order of $A^{6(k-1)}$, with an explicit leading constant given in terms of the Mirzakhani function {, and we obtain a similar result for products of nonpositively curved surfaces.} We also prove rigidity of the Lagrangian area spectrum, and obtain analogous counting results for products of a higher genus surface with a circle.
	\end{abstract}

\section{Introduction} 
	Let $\Sigma$ be a closed orientable surface of genus $g\geq 2$. Minimal Lagrangian surfaces in $\Sigma\times\Sigma$ have been extensively studied. To mention a few examples, minimal Lagrangians arise naturally in higher-rank Teichm\"{u}ller theory \cite{labourie}, in the study of three-dimensional anti--de Sitter geometry \cite{bonsante0}, in the study of minimal diffeomorphisms \cite{markovic}, and in the deformation theory of area-preserving maps \cite{wang}. Further properties were studied in \cite{ouyang}. Existence and uniqueness were established by Schoen \cite{schoen} and later extended by Lee \cite{lee}. 
	
	In this paper, we borrow ideas from dynamical systems and Teichm\"{u}ller theory to study minimal Lagrangian surfaces of large area. More precisely, we provide evidence that minimal Lagrangian surfaces behave like closed geodesics on hyperbolic surfaces. In the case where the minimal Lagrangian has genus one, this analogy is immediate, since such surfaces are finite covers of products of closed geodesics. The main challenge, and the focus of this work, is to uncover analogous phenomena for minimal Lagrangian surfaces of genus two or higher.

	A celebrated result of Mirzakhani  \cite{mirza2}  counted the growth rate of simple closed geodesics on $\Sigma$. Later, in \cite{mirza}, she extended the result to count closed geodesics with a bounded number of self-intersections.  We show that the number of $\pi_1$-injective genus $g$ minimal Lagrangians with area bounded by $A$ grows like the number of simple closed geodesics with length bounded by $A$ on a surface of genus $g$.
	
	{Another classical result in the theory of closed geodesics, proved by Croke \cite{Croke} and Otal \cite{Otal} independently, states that a negatively curved metric on a surface can be recovered from its marked {length} spectrum. This had been conjectured by Katok in the 1980's (see \cite{BurnsKatok}).  In the same vein we show that a nonpositively curved metric on a product of Riemann surfaces is uniquely determined by the area of $\pi_1$-injective minimal surfaces.}
	
\subsection{Main results}
	In order to state the main theorems we need to introduce some notation first. We will be brief but the reader can find the missing definitions in Section \ref{notation}.
\subsection{Minimal surfaces in $\Sigma\times S^1$}
	
	We start with a result about the number of minimal surfaces on $M=\Sigma\times S^1$, endowed with the product metric $\sigma\oplus d\theta^2$. Besides being of independent interest, it also serves to contextualize the  theorems regarding minimal Lagrangians.

Let $\mathcal S^*(k)$ denote the set of all surface subgroups of $\pi_1(M)$ with genus $k$ modulo the equivalence relation of conjugacy.  We abuse notation and see an element $\Pi\in \mathcal S^*(k)$ as being the set of all $\pi_1$-injective  immersions
	\begin{equation}\label{homopty.classes}
		\iota:S\rightarrow \Sigma\times S^1\quad\text{with $S$ a genus $k$ surface and }\iota_*(\pi_1(S))\in\Pi.
	\end{equation}
	We denote an immersion as above   by $(S,\iota)$.  	
	Given $\Pi\in \mathcal S^*(k)$ we define 
	$$\text{area}_{\sigma\oplus d\theta^2}(\Pi)=\inf\{\text{area}_{\iota^*(\sigma\oplus d\theta^2)}(S): (S,\iota)\in \Pi\}.$$
	It is known (using Schoen--Yau \cite{schoen-yau}, for instance) that there exists an essential genus-$k$ minimal surface in $\Pi$ which achieves $\text{area}_{\sigma\oplus d\theta^2}(\Pi)$.  If the metric on $\Sigma$ is hyperbolic, those minimal surfaces were classified by Meeks--Rosenberg in  \cite{meeks-rosenberg}. The first result follows from the growth of closed geodesics  on negatively curved surfaces.
	\begin{thm}\label{codimension.one.torus.thm} With $\sigma$ a negatively curved metric on $\Sigma$ with  topological entropy $\delta(\sigma)$ we have
		$$\lim_{A\to \infty}\frac{\ln \#\{\Pi\in \mathcal S^*(1):\text{area}_{\sigma\oplus d\theta^2}(\Pi)\leq A\}}{A}=\frac{\delta(\sigma)}{2\pi}.$$
	\end{thm}
	
	In order to study the area growth  in $\mathcal S^*(k)$ when $k>1$, we need the stable norm $||\cdot||_\sigma$ on $H_1(\Sigma,\R)$, which was introduced by Federer \cite{federer}. It has the property that if $\gamma\in H_1(\Sigma,\Z)$,  then $||\gamma||_\sigma$ is the length of the least-length  $1$-cycle in  $\gamma$. We use $B_1(\sigma)\subset H_1(\Sigma,\R)$ to denote the unit ball for the stable norm.
	It is standard to identify $H_1(\Sigma,\Z)$ with the integer lattice in $H_1(\Sigma,\R)$ obtained from the change-of-coefficients map. We use $\mu$ to denote the Lebesgue measure on $H_1(\Sigma,\R)$, normalized so that the torus $H_1(\Sigma,\R)/H_1(\Sigma,\Z)$ has volume $1$.
	
	If $k-1=d(g-1)$ we denote by $C_g(k)$ the set of equivalence classes of degree $d$ covers of $\Sigma$.  {Two covers $S \rightarrow \Sigma$ are equivalent if they are homotopic, up to precomposing with an element of the mapping class group of $S$.}  Given a finite cover $p$ of $\Sigma$, we denote by $|D(p)|$ the cardinality of the group of deck transformations of $p$.
	\begin{thm}\label{codimension.one.thm} With $\sigma$ a metric on $\Sigma$ we have
		
		$$
		\lim_{A\to\infty}\frac{\#\{\Pi\in \mathcal S^*(g):\text{area}_{\sigma\oplus d\theta^2}(\Pi)\leq A\}}{A^{2g}}=\frac{\mu(B_1(\sigma))}{(2\pi)^{2g}}.$$

		If $k-1\geq 1$ is not a positive integral multiple of $g-1$ then $\mathcal S^*(k)=\emptyset$. 
		
		If   $k-1$ is  a positive multiple of $g-1$ then
		\begin{equation*}
			\lim_{A\to\infty}\frac{\#\{\Pi\in \mathcal S^*(k):\text{area}_{\sigma\oplus d\theta^2}(\Pi)\leq A\}}{A^{2k}}=\sum_{p\in C_g(k)}\frac{\mu(B_1(p^*\sigma))}{|D(p)|(2\pi)^{2k}}.
		\end{equation*}
	\end{thm}
	
\subsection{{Lagrangian surface groups in $\Sigma\times \Sigma$}}
	Fix  a product symplectic form $$\bar \omega=\bar \omega_1-\bar\omega_2$$  on $\Sigma\times \Sigma$, where each $\bar \omega_i$ is the pullback of an area form on the  $i$-th factor with area $4\pi(g-1)$.
	
	Let $\mathcal L^*(k)$ denote the set of all Lagrangian surface subgroups of $\pi_1(\Sigma\times\Sigma)$ with genus $k$ modulo the equivalence relation of conjugacy. We abuse notation and  {regard elements} $\Pi\in \mathcal L^*(k)$ as  $\pi_1$-injective immersions
	\begin{equation}\label{homopty.classes.lagrangian}
		\iota:S\rightarrow \Sigma\times \Sigma \quad\text{with }\int_S\iota^*\bar\omega=0\quad \text{and}\quad \iota_*(\pi_1(S))\in\Pi,
	\end{equation}
	where $S$ is a genus $k$ surface. Given a metric $h$ on $\Sigma\times \Sigma$, consider
	\begin{equation}\label{area.minimizing.lagrangian}
		\text{area}_{h}(\Pi)=\inf\{\text{area}_{\iota^*(h)}(S):(S,\iota)\in \Pi\}.
	\end{equation}
 If $h$ is a nonpositively curved metric on $\Sigma\times \Sigma$ then $h$ is isometric to $\sigma\oplus\rho$ \cite[Theorem 2]{gromoll-wolf}, where $\sigma$ and $\rho$ are nonpositively curved. We study the growth of  $\text{area}_{h}(\Pi)$  following similar principles to the ones used in the previous case and in the case of simple closed geodesics in $\Sigma$. 
 
 Let $\text{sys}(\tau)$ and $\delta(\tau)$ denote, respectively, the systole  and the topological entropy of the metric $\tau$ on $\Sigma$.
	\begin{thm}\label{lagrangian.tori2}
		Consider a nonpositively curved metric $h=\sigma\oplus\rho$ on $\Sigma\times \Sigma$. We have
		{$$\lim_{A\to\infty}\frac{\ln(\#\{\Pi\in \mathcal L^*(1):\text{area}_h(\Pi)\leq A\})}{A}=\max\left\{\frac{\delta(\rho)}{\text{sys}(\sigma)},\frac{\delta(\sigma)}{\text{sys}(\rho)}\right\}.$$}
	\end{thm}
		{We use $\mathcal M_{\leq 0}(\Sigma)$ and  $\mathcal M(\Sigma)$ to denote, respectively,  the set of all smooth nonpositively curved metrics  on $\Sigma$ (modulo pull-back by diffeomorphisms homotopic to the identity) and the moduli space of hyperbolic metrics on $\Sigma$.}
		
		 In order to study the area growth in $\mathcal L^*(k)$ we need the space of measured geodesic laminations $\mathcal{ML}(\Sigma)$. Informally, $\mathcal{ML}(\Sigma)$ can be seen as a finer analogue of $H_1(\Sigma,\R)$: after we fix a metric $\tau\in \mathcal M_{\leq 0}(\Sigma)$, elements in $\mathcal{ML}(\Sigma)$ minimize length in their homotopy class while elements in $H_1(\Sigma,\R)$ minimize length in their homology class. With $\lambda_\tau$ being the Liouville current of $\tau$, we consider the function (see Section \ref{currents.section})
	$$\lambda\in \mathcal{ML}(\Sigma)\mapsto l_\tau(\lambda)=i(\lambda_\tau,\lambda).$$
	Keeping with the analogy to $H_1(\Sigma,\R)$, $l_\tau(\lambda)$ can be seen as the stable norm for the metric $\tau$ on $\mathcal{ML}(\Sigma)$.
	
	Thurston introduced a locally finite measure $m_{Th}$ on the space of measured geodesic laminations $\mathcal{ML}(\Sigma)$. It is invariant under the action by the mapping class group. The \emph{Mirzakhani function} is defined as
	$$B:{\mathcal M_{\leq 0}(\Sigma)}\rightarrow (0,\infty),\quad B(\tau)=m_{Th}(\{\lambda\in \mathcal{ML}(\Sigma):l_\tau(\lambda)\leq 1\}).$$
	Mirzakhani also considered the average of $\tau\mapsto B(\tau)$ with respect to the Weil--Petersson volume form:
	$$b_g=\int_{\mathcal M(\Sigma)}B(\tau)\,dV_{WP}(\tau).$$
	The constant $b_g$ can be explicitly computed \cite{delecroix}. With $\sigma$ being a  negatively curved metric on $\Sigma$, let $\mathcal S$ be the set of simple closed geodesics in $\Sigma$. Mirzakhani showed (see Remark on \cite[p. 101]{mirza2})
	$$\lim_{L\to\infty}\frac{\#\{\gamma\in \mathcal S:l_\sigma(\gamma)\leq L\}}{L^{6(g-1)}}=n_g\frac{B(\sigma)}{b_g},$$
	where $n_g\in \mathbb{Q}$ depends only on $g$ and can be computed in many cases (\cite[Section 11]{souto-book} and references therein).
	{\begin{thm}\label{main.thm.variable}
		Consider a nonpositively curved metric $h=\sigma\oplus\rho$ on $\Sigma\times \Sigma$. We have
		$$\lim_{A\to\infty}\frac{\#\{\Pi\in \mathcal L^*(g):\text{area}_h(\Pi)\leq A\}}{A^{6(g-1)}}=\frac{B(\sigma)B(\rho)}{b_g}.$$
		
		If $k-1\geq 1$ is not a positive integral multiple of $g-1$ then $\mathcal L^*(k)=\emptyset$.
		
		If $k-1$ is a positive multiple of $g-1$ then
		$$\lim_{A\to\infty}\frac{\#\{\Pi\in \mathcal L^*(k):\text{area}_h(\Pi)\leq A\}}{A^{6(k-1)}}=\sum_{p_1,p_2\in C_g(k)}\frac{B(p_1^*\sigma)B(p_2^*\rho)}{|D(p_1)||D(p_2)|b_k}.$$
	\end{thm}}

We {now} further assume that the metric $h$ is a  K\"{a}hler--Einstein metric normalized so that
	$$
	\operatorname{Ric}(h)=-h.
	$$
	Up to diffeomorphism\footnote{By \cite[Theorem 3.4]{catanese}, the underlying complex surface is biholomorphic to a product of two genus-\(g\) curves. The conclusion then follows from Aubin--Yau uniqueness.}, we may write
	$
	h=\sigma\oplus\rho,
	$
	where $\sigma$ and $\rho$ are hyperbolic metrics on $\Sigma$. Consider the symplectic form
	\begin{equation}\label{symplectic.form}
		\omega=\omega_\sigma-\omega_\rho.
	\end{equation}
	This form is cohomologous to $\bar\omega$. A branched immersion into $\Sigma\times\Sigma$ is called \emph{Lagrangian} if $\omega$ vanishes pointwise along the image.
	
	 Consider $\mathcal L(k)$ to be the set of all $\pi_1$-injective minimal Lagrangian immersions of a surface of genus $k\geq 1$ into $\Sigma\times \Sigma$. It essentially follows from \cite{lee} (see Section \ref{top.argument} for details) that $\mathcal L(k)$ is in one-to-one correspondence with $\mathcal L^*(k)$ and that each $L\in\mathcal L(k)$ uniquely determines $[L]\in \mathcal L^*(k)$ so that 
	 $$\text{area}_h([L])=\text{area}_h(L).$$
	{Thus we obtain the following corollary of Theorem \ref{main.thm.variable}.

	\begin{cor}[Minimal Lagrangian Counting]\label{main.thm}
		Consider a normalized K\"{a}hler--Einstein metric $h=\sigma\oplus\rho$ on $\Sigma\times \Sigma$. We have
		$$\lim_{A\to\infty}\frac{\#\{L\in \mathcal L(g):\text{area}_h(L)\leq A\}}{A^{6(g-1)}}=\frac{B(\sigma)B(\rho)}{b_g}.$$
		
		If $k-1\geq 1$ is not a positive integral multiple of $g-1$ then $\mathcal L(k)=\emptyset$.
		
		If $k-1$ is a positive multiple of $g-1$ then
		$$\lim_{A\to\infty}\frac{\#\{L\in \mathcal L(k):\text{area}_h(L)\leq A\}}{A^{6(k-1)}}=\sum_{p_1,p_2\in C_g(k)}\frac{B(p_1^*\sigma)B(p_2^*\rho)}{|D(p_1)||D(p_2)|b_k}.$$
	\end{cor}}
	Filip \cite{filip} counted the number of special Lagrangian tori on a generic $K3$-surface. Special Lagrangians are calibrated, which means their area is a homological invariant. In our setting it is possible to have homology classes which contain embedded minimal Lagrangians but for which the area minimizing representative is not Lagrangian. 

\subsection{Rigidity of area spectrum} 
Given a metric $h$ on $\Sigma\times \Sigma$, the genus $g$ Lagrangian area spectrum of $\Sigma\times \Sigma$ is the map
	$$A(h):\mathcal L^*(g)\rightarrow (0,\infty),\quad \Pi\mapsto\text{area}_{h}(\Pi).$$
	We note that while $A(\sigma\oplus\rho)$ and $A(\rho\oplus\sigma)$ have the same image, they will be different maps in general. 
	\begin{thm}[Area rigidity]\label{rigidity.thm} Consider pairs of nonpositively curved metrics $(\sigma,\rho)$ and $(\hat\sigma,\hat \rho)$ so that
		$$A(\sigma\oplus\rho)=A(\hat\sigma\oplus\hat\rho).$$
		Then there is $c_0>0$ so that $\sigma$ and $c_0\rho$ have the same simple length spectrum as $c_0\hat\sigma$ and $\hat\rho$, respectively.
	\end{thm}
	
	Using Theorem \ref{rigidity.thm} we prove the following corollaries.
	\begin{cor}[Lagrangian Area Rigidity]\label{area.rigidity}
		Consider K\"{a}hler--Einstein metrics $h=\sigma\oplus\tau$ and $\bar h=\bar\sigma\oplus \bar\tau$, not necessarily with the same Einstein constant, and so that
		$$A(h)=A(\bar h).$$
		Then $h$ and $\bar h$ are isometric and the isometry is homotopic to the identity.
	\end{cor}
	We use $\mathcal S^*(k)$ to denote the set of all surface subgroups of $\pi_1(\Sigma\times\Sigma)$ with genus $k$ modulo the equivalence relation of conjugacy. Using Theorem \ref{rigidity.thm} we deduce the following corollary.
	\begin{cor}\label{rigidity.cor} Consider $h,\hat h$ nonpositively curved metrics on $\Sigma\times \Sigma$ so that for every $k\geq g$ and $\Pi\in \mathcal S^*(k)$ we have 
		$$\text{area}_{h}(\Pi)=\text{area}_{\hat h}(\Pi).$$
		Then $h$ and $\hat h$ are isometric and the isometry is homotopic to the identity.
	\end{cor}
	On $\Sigma\times S^1$ such area rigidity as in Theorem \ref{rigidity.thm} does not hold, as we now explain.  Choose a diffeomorphism $\phi$ of $\Sigma$ representing a non-trivial element of the Torelli group, meaning that $\phi$ acts as the identity on $H_1(\Sigma,\Z)$. From the description of $\mathcal S^*(g)$ in Lemma \ref{conjugacy.one} we see that for every $\Pi\in \mathcal S^*(g)$ we have that $(\phi\times Id)_*(\Pi)=\Pi$. With $\sigma$ a hyperbolic metric, $\sigma$ and $\phi^*\sigma$ induce distinct elements in $\mathcal T(\Sigma)$ but nonetheless, we have that  for all $\Pi\in \mathcal S^*(g)$
	\begin{align*}
		\text{area}_{\phi^*\sigma\oplus d\theta^2}(\Pi)&=\text{area}_{(\phi\times Id)^*(\sigma\oplus d\theta^2)}(\Pi)= \text{area}_{\sigma\oplus d\theta^2}((\phi\times Id)_*(\Pi))\\
		&=\text{area}_{\sigma\oplus d\theta^2}(\Pi).
	\end{align*}
	
\subsection{New questions}
	Some of the classical questions {concerning} closed geodesics on {negatively curved} surfaces are their length growth, whether their lengths determine the metric, and what is the typical behavior of geodesics with large length. Corollary \ref{main.thm} and Theorem \ref{rigidity.thm} answer analogues of the first two questions.  
	
	There is a finite set of simple closed geodesics (with $9(g-1)$ elements) whose (marked) lengths uniquely determine an element in Teichm\"{u}ller space (see {the} $9g-9$ Theorem in \cite{farb-margalit}).  In Theorem~2.1 of \cite{mcshane}, the authors find a finite set $S_0$ of simple closed geodesics of $\Sigma$ so that if 
	$$\text{area}_{\sigma\oplus\tau}([\alpha\times \beta])=\text{area}_{\bar\sigma\oplus \bar\tau}([\alpha\times \beta])\quad\text{for all }\alpha,\beta\in S_0,$$
	then $\sigma$ and $\tau$ induce the same element in $\mathcal T(\Sigma)$ as $\bar\sigma$ and $\bar\tau$, respectively.
	
	\begin{conj} There is a finite set $\mathcal L_0\subset \mathcal L^*(g)$ so that if $h$ and $\hat h$ are K\"{a}hler--Einstein metrics on $\Sigma\times\Sigma$ with
		$$A(h)(\Pi)=A(\hat h)(\Pi)\quad\text{for all }\Pi\in\mathcal L_0,$$
		then $h$ and $\hat h$ are isometric and the isometry is homotopic to the identity.
	\end{conj}
	This would mean that the area of finitely many minimal Lagrangians  determines the  K\"ahler--Einstein metric.  In Section \ref{evidence.conjecture} we give some evidence towards the conjecture. Fixing $\sigma_0\in \mathcal T(\Sigma)$ and identifying $\mathcal{ML}(\sigma_0)$ (see \eqref{PML}) with the boundary of Teichm\"uller space, so that $\overline{\mathcal T(\Sigma)}={\mathcal T(\Sigma)}\cup \mathcal{ML}(\sigma_0)$, there is a {natural} way to renormalize $(\sigma,\rho)\mapsto A(\sigma\oplus\rho)(\Pi)$ so that it extends continuously to $\overline{\mathcal T(\Sigma)}\times \overline{\mathcal T(\Sigma)}$ (see \cite[Section 8]{BMS} for a related renormalization).  In Theorem \ref{finiteness.rigidity.thm} we check that the (analogous) conjecture holds on $\mathcal{ML}(\sigma_0)\times \mathcal{ML}(\sigma_0)$.

	Regarding the behavior of minimal surfaces with large area we propose the following problem. 
	
	An element $L\in  \mathcal L(k)$ is {\em primitive} if there is no finite cover map from $L$ to some $L'\in  \mathcal L(k')$ with $k'< k$. Consider the subset $\mathcal L\subset \cup_{k\geq g}\mathcal L(k)$ of those elements that are primitive. Given $A$, consider the finite set $\mathcal I(A-1,A)$ of all elements  in $\mathcal L$ whose area is in $(A-1,A)$.  With $LGr(2,4)$ denoting the bundle of Lagrangian planes in $T(\Sigma\times \Sigma)$ we obtain a measure
	$$\mu_A(f)=\frac{1}{A\# \mathcal I(A-1,A)}\sum_{L\in\mathcal I (A-1,A)}\int_Lf(x,T_xL)dA_{L}(x),\, f\in C^0(LGr(2,4)).$$
	\begin{prob}  Describe the weak-* subsequential limits of $\mu_A$ as $A\to\infty$.
	\end{prob}
	The limit of $\mu_A$ will give information regarding a typical minimal Lagrangian with large area. %The naive conjecture would be that the limit becomes equidistributed. 
	
	%A natural question is to extend Theorem \ref{main.thm} to the case where the ambient metric is only assumed to be nonpositively curved. The lower bound for the area given in Proposition \ref{prop2} also holds for nonpositive  curvature but the upper bound seems to require a new argument  in the context of variable curvature.

\subsection{Acknowledgments} 
	
	B.L. thanks Aaron Calderon for a helpful conversation.

\section{Proof of Theorems in $\Sigma\times S^1$}
	
\subsection{Proof of Theorem \ref{codimension.one.torus.thm}}
	We use $\mathcal C$ to denote the set of primitive closed $\sigma$-geodesics on $\Sigma$. 
	
	\begin{lemm}\label{torus.cover} Given $\Pi\in\mathcal S^*(1)$  there is a unique $\alpha=\alpha(\Pi)\in \mathcal C$ and $d=d(\Pi)\in \N$ so that  $\text{area}_{\sigma\oplus d\theta^2}(\Pi)=2\pi d\,l_\sigma(\alpha)$. Moreover, given $d\in\N$ and $\alpha\in \mathcal C$ we have
		$\#\{\Pi\in \mathcal S^*(1): \alpha(\Pi)=\alpha, d(\Pi)=d\}\leq d^2.$
	\end{lemm}
	\begin{proof}
		Fix $\Pi\in \mathcal S^*(1)$ and consider 
		$\iota:T^2\rightarrow \Sigma\times S^1$
		to be  a conformal area-minimizing  map in $\Pi$. The map is harmonic, the ambient space has nonpositive curvature, and therefore Eells--Sampson \cite{eells-sampson} implies that $\iota$ is totally geodesic. The map is $\pi_1$-injective and so  there is a unique  $\alpha=\alpha(\Pi)\in \mathcal C$ so that $\iota(T^2)= \alpha\times S^1.$ Furthermore, $\iota_*(\Z^2)$ is a rank-two subgroup of $\pi_1(\alpha\times S^1)$ and has thus finite index. Therefore $\iota:T^2\rightarrow \alpha\times S^1$ is a finite covering map.
		Denote its degree by $d=d(\Pi)$. We conclude that
		$$\text{area}_{\sigma\oplus d\theta^2}(\Pi)=2\pi d\,l_\sigma(\alpha).$$
		
		For the counting statement, fix $\alpha\in \mathcal C$ and $d\in \N$. Every $\Pi$ with $\alpha(\Pi)=\alpha$ and $d(\Pi)=d$ is represented by a connected degree-$d$ cover of the torus $\alpha\times S^1$. The number of such covers is the number of index-$d$ subgroups of $\pi_1(\alpha\times S^1)\cong \Z^2$, which is known to be at most $d^2$.
	\end{proof}
	
	\begin{thm} With $\sigma$ a negatively curved metric on $\Sigma$ with  topological entropy $\delta(\sigma)$ we have
		$$\lim_{A\to \infty}\frac{\ln \#\{\Pi\in \mathcal S^*(1):\text{area}_{\sigma\oplus d\theta^2}(\Pi)\leq A\}}{A}=\frac{\delta(\sigma)}{2\pi}.$$
	\end{thm}	
	\begin{proof}	
		There is $c_0=c_0(\sigma)$ so that if $2\pi d\,l_\sigma(\alpha)\leq A$ then $d<c_0A$.  If we set $N=\lfloor c_0A \rfloor$, we have from Lemma \ref{torus.cover} that
		$$\{\Pi\in \mathcal S^*(1):\text{area}_{\sigma\oplus d\theta^2}(\Pi)\leq A\}\subset \bigcup_{d=1}^N\{\Pi\in \mathcal S^*(1):2\pi l_\sigma(\alpha(\Pi))\leq A, d(\Pi)=d\}.$$
		Using Lemma \ref{torus.cover} again we see that
		$$\#\{\Pi\in \mathcal S^*(1):\text{area}_{\sigma\oplus d\theta^2}(\Pi)\leq A\}\leq c_0^3A^3\#\{\alpha\in \mathcal C:2\pi l_\sigma(\alpha)\leq A\}.$$
		Hence, using Margulis \cite{margulis} in the last equality, we obtain that	
		\begin{multline*}
			\limsup_{A\to \infty}\frac{\ln \#\{\Pi\in \mathcal S^*(1):\text{area}_{\sigma\oplus d\theta^2}(\Pi)\leq A\}}{A}\\
			\leq \lim_{A\to \infty}\frac{\ln \#\{\alpha\in \mathcal C:2\pi l_\sigma(\alpha)\leq A\}}{A}=\frac{\delta(\sigma)}{2\pi}.
		\end{multline*}
		The lower bound follows in the same way because
		$$\#\{\alpha\in \mathcal C:2\pi l_\sigma(\alpha)\leq A\}\leq \#\{\Pi\in \mathcal S^*(1):\text{area}_{\sigma\oplus d\theta^2}(\Pi)\leq A\}.$$
	\end{proof}	
\subsection{Proof of Theorem \ref{codimension.one.thm}} 
	Consider  a surface $S$ of genus $k>1$. Fix a finite cover from $S$ to $\Sigma$ in each equivalence class of $C_g(k)$, so that $C_g(k)=\{p_i\}_{i=1}^n$. If for some homeomorphism $T$ of $S$ we have $p_i\circ T$ homotopic to $p_j$, then $p_i=p_j$ and $T$ is homotopic to an element of $D(p_i)$ (recall that $D(p_i)$ is the deck group of $p_i$).
	
	Given $\gamma\in H_1(S,\Z)$ fix  a map $\phi(\gamma):S\rightarrow S^1$  representing the cohomology class of $PD(\gamma)\in H^1(S,\Z)$.  If $T$ is an orientation-preserving homeomorphism  of $S$ and  $\phi(\beta)\circ T$ is homotopic to $\phi(\gamma)$, then $T^*PD(\beta)=PD(\gamma)$ and so $\beta=T_*(\gamma)$.

	Given $\gamma\in H_1(S,\Z)$ and $p\in C_g(k)$ consider the $\pi_1$-injective map
	$$\iota(p,\gamma):S\rightarrow \Sigma\times S^1,\quad \iota(p,\gamma)(x)=(p(x),\phi(\gamma)(x)).$$
	If $T$ is a deck transformation of $p$ then $\iota(p,\gamma)$ and $\iota(p,T_*\gamma)\circ T$ are homotopic maps. When $g=k$ we use $\iota(\gamma)$ instead of $\iota(\text{Id},\gamma)$.
	\begin{lemm}\label{conjugacy.one}
		Consider a  surface  $S$ of genus $k>1$. If $k-1$ is not a positive multiple of $g-1$ then $\mathcal S^*(k)=\emptyset$. 
		
		If $k-1$ is a positive multiple of $g-1$ then the assignment 
		$$(p,\gamma)\in  C_g(k)\times H_1(S,\Z)\mapsto [\iota(p,\gamma)_*(\pi_1(S))]\in \mathcal S^*(k)$$
		is surjective. We have   $[\iota(p,\gamma)_*(\pi_1(S))]= [\iota(q,\beta)_*(\pi_1(S))]$ if and only if  $p=q$ and $\beta=T_*(\gamma)$ for some $T\in D(p)$.
	\end{lemm}
	\begin{proof}
		Suppose we are given a $\pi_1$-injective map
		$$
		\iota:S\rightarrow \Sigma\times S^1,\quad \iota=(\iota_1,\iota_2),
		$$
		where $S$ has genus $k$. Let $G=\ker (\iota_1)_*$. Since $\iota$ is $\pi_1$-injective, the restriction of $(\iota_2)_*$ to $G$ is injective. Because $\pi_1(S^1)=\Z$, the group $G$ is cyclic. On the other hand, $G$ is normal in $\pi_1(S)$, and a surface group of genus $k>1$ has no nontrivial cyclic normal subgroup. Hence $G=\{1\}$, so $(\iota_1)_*$ is injective. By \cite{scott}, any finitely generated subgroup of $\pi_1(\Sigma)$ is either free or of finite index. Since $(\iota_1)_*(\pi_1(S))\cong \pi_1(S)$ is not free, it has finite index in $\pi_1(\Sigma)$. Therefore $\iota_1$ is homotopic to a covering map of degree $d>0$. Consequently $\chi(S)=d\chi(\Sigma)$, and thus $\mathcal S^*(k)=\emptyset$ if $k-1$ is not a positive multiple of $g-1$.
		
		We assume that  $k-1=d(g-1)$, $d\in\N$. After precomposing $\iota$ with a homeomorphism of $S$ we see that $\iota_1$ is homotopic to  some  $p\in C_g(k)$. After homotoping $\iota_2$ to a smooth map, choose $\theta\in S^1$ to be a regular value and set $\gamma=[\iota_2^{-1}(\theta)]\in H_1(S,\Z)$. The homotopy class of $\iota_2$  is uniquely determined by $PD(\gamma)$ and so $\iota_2$ is homotopic to $\phi(\gamma)$. Thus $\iota$ is homotopic to $\iota(p,\gamma)$ and therefore every element of $\mathcal S^*(k)$ is represented by some $\iota(p,\gamma)$.
		
		Given $a\in \pi_1(\Sigma\times S^1)$, denote by $C(a)$ the conjugation map in $\pi_1(\Sigma\times S^1)$. Suppose that $\iota(p,\gamma)_*(\pi_1(S))$ is conjugate to $\iota(q,\beta)_*(\pi_1(S))$. Thus we can find $a\in \pi_1(\Sigma\times S^1)$ so that $f=\iota(q,\beta)_*^{-1}\circ C(a)\circ \iota(p,\gamma)_*$ is an  automorphism of $\pi_1(S)$, and hence defines an element of $\mathrm{Out}(\pi_1(S))$. By the Dehn--Nielsen--Baer theorem (see \cite[Chapter~8]{farb-margalit}) there is a homeomorphism $T$ of $S$ representing this outer automorphism. Therefore $(\iota(q,\beta)\circ T)_*$ is conjugate to $C(a)\circ \iota(p,\gamma)_*$, and hence $\iota(q,\beta)\circ T$ is homotopic to $\iota(p,\gamma)$. Hence $q\circ T$ is homotopic to $p$ and $\phi(\beta)\circ T$ is homotopic to $\phi(\gamma)$.   This implies $q=p$ and that $T$ is homotopic to some $T_0\in D(p)$. In particular, $T$ is orientation-preserving, and so $\beta=T_*(\gamma)=T_{0*}(\gamma)$.
	\end{proof}
	
\begin{thm}With $\sigma$ a metric on $\Sigma$ we have
		
		\begin{equation}\label{asymptotic.stable}
			\lim_{A\to\infty}\frac{\#\{\Pi\in \mathcal S^*(g):\text{area}_{\sigma\oplus d\theta^2}(\Pi)\leq A\}}{A^{2g}}=\frac{\mu(B_1(\sigma))}{(2\pi)^{2g}}.\end{equation}

		If $k-1\geq 1$ is not a positive integral multiple of $g-1$ then $\mathcal S^*(k)=\emptyset$. 
		
		If   $k-1$ is  a positive multiple of $g-1$ then
		\begin{equation*}
			\lim_{A\to\infty}\frac{\#\{\Pi\in \mathcal S^*(k):\text{area}_{\sigma\oplus d\theta^2}(\Pi)\leq A\}}{A^{2k}}=\sum_{p\in C_g(k)}\frac{\mu(B_1(p^*\sigma))}{|D(p)|(2\pi)^{2k}}.
		\end{equation*}
	\end{thm}
	\begin{proof}
		
		From Lemma \ref{conjugacy.one} we see that if $k-1\geq 1$ is not a positive integral multiple of $g-1$ then $\mathcal S^*(k)=\emptyset$.  We assume  first that  $k=g$ and set $\Pi(\gamma)= [\iota(\gamma)_*(\pi_1(\Sigma))]\in \mathcal S^*(g)$, $\gamma\in H_1(\Sigma,\Z)$.  This assignment is a bijection by Lemma \ref{conjugacy.one}.

		\begin{prop}\label{area.estimate} For every metric $\sigma$  on $\Sigma$ and $\gamma\in H_1(\Sigma,\Z)$
			$$2\pi||\gamma||_\sigma\leq \text{area}_{\sigma\oplus d\theta^2}(\Pi(\gamma))\leq 2\pi||\gamma||_\sigma+\text{area}_\sigma(\Sigma).$$
		\end{prop}
		\begin{proof}
			Given $(\Sigma,\iota)\in \Pi(\gamma)$ write $\iota=(\iota_1,\iota_2)$. From the proof of Lemma \ref{conjugacy.one} we see that, without loss of generality, we can assume $\iota_1$ is homotopic  to the identity and that  $\iota_2$ is homotopic to $\phi(\gamma)$. Thus for almost every $\theta\in S^1$ we have $$[\iota_2^{-1}(\theta)]=[\phi(\gamma)^{-1}(\theta)]=\gamma\in H_1(\Sigma,\Z).$$ 
			
			Set $\tau=\iota^*(\sigma\oplus d\theta^2)$. We have $|d\iota_2|_\tau\leq1$ and so we obtain from the co-area formula and the fact that along $\iota_2^{-1}(\theta)$ the metric $\tau$ agrees with $\iota_1^*\sigma$ that
			\begin{align*}
				\text{area}_{\iota^*(\sigma\oplus d\theta^2)}(\Sigma)&=\int_\Sigma dA_\tau\geq \int_\Sigma |d\iota_2|_\tau dA_\tau= \int_{S^1}\text{length}_{\tau}(\iota_2^{-1}(\theta))d\theta\\
				&=\int_{S^1}\text{length}_{\sigma}(\iota_1(\iota_2^{-1}(\theta)))d\theta \geq 2\pi||\gamma||_\sigma.
			\end{align*}
			Taking the infimum over all $(\Sigma,\iota)\in \Pi(\gamma)$ gives the lower bound.

			If $\gamma=0$, the map $x\mapsto (x,1)$ belongs to $\Pi(\gamma)$ and already gives the upper bound. We assume from now on that $\gamma\neq 0$. Choose pairwise disjoint simple closed geodesics $c_1,\ldots,c_n$ and integers $a_1,\ldots,a_n\geq 1$ so that
			$$
			\gamma=\sum_{j=1}^na_j[c_j]\quad\text{and}\quad ||\gamma||_\sigma=\sum_{j=1}^n a_jl_\sigma(c_j).
			$$
			Given $\varepsilon>0$, set $N=\sum_{j=1}^na_j$ and choose pairwise disjoint tubular neighborhoods $U_j$ of $c_j$ with orientation-preserving parametrizations
			$$
			F_j:(-1,1)\times S^1\to U_j,\quad F_j(0,S^1)=c_j,
			$$
			so that $\theta\mapsto  F_j(0,\theta)$ preserves the orientation of $c_j$ and for every $s\in (-1,1)$ the curve $F_j(s,S^1)$ has $\sigma$-length at most $l_\sigma(c_j)+\varepsilon/N$.
			
			Consider $\eta:[0,2]\to[0,1]$ a smooth non-decreasing function which is equal to  $0$ near $1$ and  $1$ near $2$. We now define a smooth map $\phi_\varepsilon:\Sigma\to S^1$ to be $1$ on $\Sigma\setminus \cup_jU_j$ and, on $U_j$, to be
			$$
			\phi_\varepsilon(F_j(s,t))=e^{2\pi i a_j\eta(s+1)}.
			$$
			Every regular level of $\phi_\varepsilon$ is then the union, over $j$, of $a_j$ curves of the form $F_j(s,S^1)$. Hence every regular level curve represents $\gamma$ and so $\phi_\varepsilon$ is homotopic to  $\phi(\gamma)$. Furthermore,  every regular level curve of $\phi_\varepsilon$ has $\sigma$-length at most
			$$
			\sum_{j=1}^na_j\left(l_\sigma(c_j)+\frac{\varepsilon}{N}\right)=||\gamma||_\sigma+\varepsilon.
			$$
			Using the co-area formula again,
			$$
			\int_\Sigma |d\phi_\varepsilon|_\sigma dA_\sigma=\int_0^{2\pi}\text{length}_{\sigma}(\phi_\varepsilon^{-1}(\theta))d\theta\leq 2\pi(||\gamma||_\sigma+\varepsilon).
			$$
			From $\iota_\varepsilon=(Id,\phi_\varepsilon)\in \Pi(\gamma)$ we obtain
			\begin{align*}
				\text{area}_{\iota_\varepsilon^*(\sigma\oplus d\theta^2)}(\Sigma)&=\int_\Sigma\sqrt{1+|d\phi_\varepsilon|_\sigma^2}dA_\sigma
				\leq \text{area}_\sigma(\Sigma)+\int_\Sigma|d\phi_\varepsilon|_\sigma dA_\sigma\\
				&
				\leq \text{area}_\sigma(\Sigma)+2\pi(||\gamma||_\sigma+\varepsilon).
			\end{align*}
			Letting $\varepsilon\to 0$ gives the upper bound.
		\end{proof}
		
		We now prove \eqref{asymptotic.stable}. We have from Proposition \ref{area.estimate} that
		\begin{multline}\label{inclusions.stable}
			\#\{\gamma\in H_1(\Sigma,\Z):2\pi||\gamma||_\sigma+\text{area}_\sigma(\Sigma)\leq  A\}\\
			\leq \#\{\Pi\in \mathcal S^*(g):\text{area}_{\sigma\oplus d\theta^2}(\Pi)\leq  A\}\\
			\leq \#\{\gamma\in H_1(\Sigma,\Z):2\pi ||\gamma||_\sigma\leq A\}.
		\end{multline}
		Denote by $B_R(0)\subset H_1(\Sigma,\R)$ the ball of radius $R$ for the stable norm. From scaling we have $\mu(B_R(0))=R^{2g}\mu(B_1(\sigma))$.  {(Recall that the measure $\mu$ was defined in the introduction following Theorem \ref{codimension.one.torus.thm}.)} A classical lattice-point-counting argument implies that
		$$
		\lim_{A\to\infty}\frac{\#\{\gamma\in H_1(\Sigma,\Z):2\pi ||\gamma||_\sigma\leq A\}}{A^{2g}}=\lim_{A\to\infty}\frac{\mu(B_{(A/2\pi)}(0))}{A^{2g}}
		=\frac{\mu(B_1(\sigma))}{(2\pi)^{2g}}.
		$$
		Likewise 	
		$$
		\lim_{A\to\infty}\frac{\#\{\gamma\in H_1(\Sigma,\Z):2\pi||\gamma||_\sigma+\text{area}_\sigma(\Sigma)\leq  A\}}{A^{2g}}
		=\frac{\mu(B_1(\sigma))}{(2\pi)^{2g}}
		$$
		and so we obtain from \eqref{inclusions.stable} that
		$$
		\lim_{A\to\infty}\frac{\#\{\Pi\in \mathcal S^*(g):\text{area}_{\sigma\oplus d\theta^2}(\Pi)\leq A\}}{A^{2g}}=\frac{\mu(B_1(\sigma))}{(2\pi)^{2g}}.$$
		The next lemma finishes the proof of Theorem	\ref{codimension.one.thm}.
		
		\begin{lemm}\label{second.stable.inequality} If   $k-1$ is  a positive multiple of $g-1$ then
			$$
			\lim_{A\to\infty}\frac{\#\{\Pi\in \mathcal S^*(k):\text{area}_{\sigma\oplus d\theta^2}(\Pi)\leq A\}}{A^{2k}}=\sum_{p\in C_g(k)}\frac{\mu(B_1(p^*\sigma))}{|D(p)|(2\pi)^{2k}}.
			$$
		\end{lemm}
		\begin{proof}
			For $p\in C_g(k)$ let $\mathcal S_p^*(k)\subset \mathcal S^*(k)$ be the subset consisting of those elements represented by $\iota(p,\gamma)$, $\gamma\in H_1(S,\Z)$. By Lemma \ref{conjugacy.one} we have the disjoint union
			$$
			\mathcal S^*(k)=\bigcup_{p\in C_g(k)}\mathcal S_p^*(k),
			$$
			and for each $p\in C_g(k)$ the assignment $\gamma\mapsto [\iota(p,\gamma)_*(\pi_1(S))]$ induces a bijection between $\mathcal S_p^*(k)$ and the set of $D(p)$-orbits in $H_1(S,\Z)$.
			
			Fix $p\in C_g(k)$. Consider the local isometry between $\sigma\oplus d\theta^2$ and $p^*\sigma\oplus d\theta^2$
			$$
			F_p:S\times S^1\rightarrow \Sigma\times S^1,\quad F_p(x,\theta)=(p(x),\theta),
			$$
			and for $\gamma\in H_1(S,\Z)$ let $\hat\iota(\gamma):S\to S\times S^1$ be given by
			$$
			\hat\iota(\gamma)(x)=(x,\phi(\gamma)(x)).
			$$
			Then $F_p\circ \hat\iota(\gamma)=\iota(p,\gamma)$. Let $\hat\Pi(\gamma)$ denote the conjugacy class induced by $\hat\iota(\gamma)$.
			
			Given $(S,\iota)\in [\iota(p,\gamma)_*(\pi_1(S))]$,  we saw in the proof of Lemma \ref{conjugacy.one} the existence of a diffeomorphism $T$ so that $\iota\circ T\simeq (p,\phi(\gamma))$. Thus we can find
			$\hat \iota\in \hat\Pi(\gamma)$ so that $F_p\circ \hat\iota=\iota\circ T$. Likewise, if $\hat \iota\in \hat\Pi(\gamma)$ then $F_p\circ \hat\iota$ is in $[\iota(p,\gamma)_*(\pi_1(S))]$. For this reason
			\begin{equation}\label{isometry.same}
				\text{area}_{\sigma\oplus d\theta^2}([\iota(p,\gamma)_*(\pi_1(S))])=\text{area}_{p^*\sigma\oplus d\theta^2}(\hat\Pi(\gamma)).
			\end{equation}
			Given $A>0$, set
			\begin{align*}
				E_p(A)&=\{\gamma\in H_1(S,\Z):\text{area}_{\sigma\oplus d\theta^2}([\iota(p,\gamma)_*(\pi_1(S))])\leq A\},\\
				N_p(A)&=\{\Pi\in \mathcal S^*_p(k):\text{area}_{\sigma\oplus d\theta^2}(\Pi)\leq A\}.
			\end{align*}
			Every element of $D(p)$ is represented by an isometry of $p^*\sigma$. Thus we see from \eqref{isometry.same} that the set $E_p(A)$ is invariant under the action of $D(p)$. 
			
			Although Proposition \ref{area.estimate} is stated for $(\Sigma,\sigma)$, the same statement applies to $(S,p^*\sigma)$. Therefore, arguing exactly as in the proof of \eqref{asymptotic.stable}, we obtain
			\begin{align*}
				\lim_{A\to\infty}\frac{\# E_p(A)}{A^{2k}}
				&=\lim_{A\to\infty}\frac{\#\{\gamma\in H_1(S,\Z):\text{area}_{p^*\sigma\oplus d\theta^2}(\hat\Pi(\gamma))\leq A\}}{A^{2k}}\\
				&=\frac{\mu(B_1(p^*\sigma))}{(2\pi)^{2k}}.
			\end{align*}
			
			Lemma \ref{conjugacy.one} implies that $\#N_p(A)$ is equal to the number of $D(p)$-orbits in $E_p(A)$. By  Burnside's lemma,
			\[
			\#N_p(A)=\frac{1}{|D(p)|}\sum_{T\in D(p)}\#E_p^T(A),\qquad E_p^T(A)=\{\gamma\in E_p(A):T_*(\gamma)=\gamma\}.
			\]
			
			We have $\#E_p^{Id}(A)=\#E_p(A)$. Suppose now that $T\neq Id$. We argue that $E_p^T(A)$ is contained in a vector space of dimension lower than $2k$ and hence its density will be $o(A^{2k})$. Let $m$ be the order of $T$ and let
			$$
			q:S\rightarrow S/\langle T\rangle
			$$
			be the quotient map. Since $T$ is a non-trivial deck transformation, the action of $\langle T\rangle$ on $S$ is free and $S/\langle T\rangle$ has genus $r<k$, with $k-1=m(r-1)$. Set
			$$
			V_T=\ker(T_*-Id)\subset H_1(S,\R).
			$$
			Let
			$
			\tau:H_1(S/\langle T\rangle,\R)\rightarrow H_1(S,\R)
			$
			be the transfer map. By \cite[\S3.G]{hatcher} (the same discussion applies in homology), $\tau$ is injective, and its image is precisely $V_T$. Hence
			$
			\dim V_T=2r<2k.
			$
			
			Because $T_*$ preserves $H_1(S,\Z)$, the subgroup
			$
			\Lambda_T=V_T\cap H_1(S,\Z)
			$
			is a lattice in $V_T$. Let $\mu_T$ be the Lebesgue measure on $V_T$, normalized so that the torus $V_T/\Lambda_T$ has volume $1$. The restriction of $||\cdot||_{p^*\sigma}$ to $V_T$ is a norm. Since $E_p^T(A)\subset \Lambda_T$, Proposition \ref{area.estimate}, applied to $(S,p^*\sigma)$, gives
			$$
			E_p^T(A)\subset \{\gamma\in \Lambda_T:2\pi||\gamma||_{p^*\sigma}\leq A\}.
			$$
			Applying the same lattice-point-counting argument as in the proof of \eqref{asymptotic.stable}, but now in the vector space $V_T$ with lattice $\Lambda_T$ and measure $\mu_T$, we obtain
			\[
			\#E_p^T(A)=O(A^{\dim V_T})=O(A^{2r})=o(A^{2k}).
			\]
			Dividing the Burnside formula by $A^{2k}$, using  $E_p^{Id}(A)=E_p(A)$, and letting $A\to\infty$, we obtain
			\[
			\lim_{A\to\infty}\frac{\#N_p(A)}{A^{2k}}=\lim_{A\to\infty}\frac{\#E_p(A)+o(A^{2k})}{|D(p)|A^{2k}}=\frac{\mu(B_1(p^*\sigma))}{|D(p)|(2\pi)^{2k}}.
			\]
			Summing over the disjoint union $\mathcal S^*(k)=\cup_{p\in C_g(k)}\mathcal S_p^*(k)$ yields the result.
		\end{proof}
	\end{proof}
	
\section{Preliminaries {for the $\Sigma\times \Sigma$ case}}\label{notation}

	Consider $\text{Diff}_+(\Sigma)$ (resp.\ $\text{Diff}_+(\Sigma)_0$), the group of all orientation-preserving diffeomorphisms (resp.\ those homotopic to the identity). The mapping class group of $\Sigma$ is
	\begin{equation}\label{mg.defi}
		\text{Map}(\Sigma)=\text{Diff}_+(\Sigma)/\text{Diff}_+(\Sigma)_0.
	\end{equation}
	{The space of all smooth nonpositively curved metrics on $\Sigma$, modulo pull-back by $\text{Diff}_+(\Sigma)_0,$ is denoted by $\mathcal M_{\leq 0}(\Sigma)$.} The Teichm\"uller space $\mathcal T(\Sigma)$ denotes the space of all hyperbolic metrics on $\Sigma$ modulo $\text{Diff}_+(\Sigma)_0$. The moduli space $\mathcal M(\Sigma)$ denotes the space of all hyperbolic metrics on $\Sigma$ modulo the action of $\text{Diff}_+(\Sigma)$. Both are finite-dimensional spaces of dimension $6g-6$. We often identify  a nonpositively curved metric $\tau$, a hyperbolic metric $\sigma$, and $\phi\in\text{Diff}_+(\Sigma)$  with its equivalence class in $\mathcal M_{\leq 0}(\Sigma)$, $\mathcal T(\Sigma)$, and  $\text{Map}(\Sigma)$.
	
	Given $\phi\in\text{Map}(\Sigma)$ and $\tau\in\mathcal M_{\leq 0}(\Sigma)$, $\phi^*\tau$ and $\phi_*\tau=(\phi^{-1})^*\tau$ are well-defined elements of $\mathcal M_{\leq 0}(\Sigma)$. This induces left and right actions of $\text{Map}(\Sigma)$ on $\mathcal T(\Sigma)$.	
	
	Let $\mathcal S$ denote the set of conjugacy classes of non-trivial simple closed curves in $\Sigma$. Given a negatively curved metric on $\Sigma$, $\mathcal S$ is in one-to-one correspondence with the set of simple closed geodesics in that metric. For this reason, we often abuse notation and refer to an element of $\mathcal S$ as either a simple closed geodesic $\gamma$ or its conjugacy class $[\gamma]$. The simple length spectrum of a  metric $\sigma$ on $\Sigma$ is the function
	$$
	l_\sigma:\mathcal S\rightarrow (0,\infty),\quad
	l_\sigma([\gamma])=\inf\{l_\sigma(c): c \text{ is homotopic to }\gamma\}.
	$$
	When the metric is negatively curved, $l_\sigma([\gamma])$ is realized by a unique geodesic in $[\gamma]$.
	The group $\text{Map}(\Sigma)$ acts naturally on $\mathcal S$ via $[\gamma]\mapsto [\phi(\gamma)]$.
	
	We use $\hat \Sigma$ to denote the universal cover of $\Sigma$ and $\pi:\hat \Sigma\rightarrow \Sigma$ the covering map.
	
\subsection{Geodesic currents}\label{currents.section}
	The reference is \cite{bonahon2}. {Fix a hyperbolic metric on $\Sigma$. The spaces constructed are independent of the hyperbolic metric chosen.}
	
	Let $G(\hat \Sigma)$ denote the set of all complete geodesics in $\hat \Sigma$ with the compact-open topology. The group $\pi_1(\Sigma)$ acts naturally on $G(\hat \Sigma)$. The space of geodesic currents $\mathcal C(\Sigma)$ consists of all locally finite $\pi_1(\Sigma)$-invariant Radon measures on $G(\hat \Sigma)$. We consider the weak$^*$ topology on $\mathcal C(\Sigma)$.
	
	Given a closed geodesic $\gamma$ on $\Sigma$, $\delta_\gamma$ denotes the counting measure  {supported on  the set of lifts of $\gamma$ to $\hat \Sigma$, counted with multiplicity.} The current $\delta_\gamma$ depends only on the conjugacy class $[\gamma]$ of $\gamma$.  {We have from  Proposition~2 in \cite{bonahon2}
	\begin{equation}\label{prop2.bonahon}
	\mathcal C(\Sigma)=\overline{\{t\delta_\gamma:\,t>0,\, \gamma \text{ a closed geodesic}\}}.
	\end{equation}}
There is a continuous bilinear pairing, called {\em intersection},
	$$
	i:\mathcal C(\Sigma)\times \mathcal C(\Sigma)\to[0,\infty),
	$$
	such that $i(\delta_\gamma,\delta_\beta)$ is equal to the unsigned count, {with multiplicity}, of intersection points $\#(\gamma\cap \beta)$ when  $\gamma$ is transverse to $\beta$ and $\gamma$ and $\beta$ minimize the number of intersection points (counted without sign) over all curves in their respective {free homotopy} classes that intersect  transversely. Also, $i(\delta_\gamma,\delta_\gamma)=0$ if and only if  the image of the closed geodesic $\gamma$ is simple.
	
	Each negatively curved metric $\rho$ gives rise to a current $\lambda_\rho \in \mathcal C(\Sigma)$, called the {\em Liouville current} associated to $\rho$. The assignment $\rho\mapsto\lambda_\rho$ is invariant under the action of $\text{Diff}_+(\Sigma)_0$. Its defining property is that, if $\gamma$ is a closed geodesic, then
	\[
	i(\lambda_\rho,\delta_{\gamma}) = l_\rho([\gamma]).
	\]
	In \cite{croke2}, the definition of $\lambda_\rho$ was extended to nonpositively curved metrics in a way that the above defining property still holds.
	Accordingly, for any $\mu \in \mathcal C(\Sigma)$, it is customary to use the notation
	\[
	l_\rho(\mu) = i(\lambda_\rho,\mu).
	\]
	{If $\rho$ is hyperbolic, then} $i(\lambda_\rho,\lambda_\rho)=\pi^2|\chi(\Sigma)|$ \cite[Proposition 15]{bonahon2}.

	The group $\text{Map}(\Sigma)$ induces an action on $\mathcal C(\Sigma)$ via pushforward of measures \cite[Chapter 3]{souto-book}. It has the property that $\phi_*\delta_{[\beta]}=\delta_{[\phi(\beta)]}$ for every conjugacy class $[\beta]$. It is customary to abuse notation and denote $\phi_*$ simply by $\phi$ and $\delta_{[\beta]}$ by $\delta_\beta$. The intersection pairing between geodesic currents is preserved by this action. From these properties we have $\phi(\lambda_\tau)=\lambda_{\phi_*\tau}$ for all nonpositively curved metrics $\tau$ and $\phi\in \text{Map}(\Sigma)$. Indeed, for every conjugacy class $[\beta]$ we have
	\begin{align*}
		i(\phi(\lambda_\tau),\delta_{\beta})
		&=i(\lambda_\tau,\phi^{-1}(\delta_{\beta}))
		= i(\lambda_\tau,\delta_{[\phi^{-1}(\beta)]})
		=l_\tau([\phi^{-1}(\beta)])\\
		&=l_{\phi_*\tau}([\beta])
		=i(\lambda_{\phi_*\tau},\delta_{\beta}).
	\end{align*}
	
	The set of all {\em measured geodesic laminations} is denoted by
	\begin{equation}\label{ML.defi}
		\mathcal{ML}(\Sigma)=\{\lambda\in\mathcal C(\Sigma): i(\lambda,\lambda)=0\}.
	\end{equation}
	There is a canonical homeomorphism (i.e., metric-independent) between $\mathcal{ML}(\Sigma)$ and the space of measured foliations \cite[Proposition 17]{bonahon2}. From this we can derive several consequences.
	
	The first is that if $\mu,\nu\in\mathcal{ML}(\Sigma)$ satisfy $i(\mu,\delta_\gamma)=i(\nu,\delta_\gamma)$ for every simple closed geodesic $\gamma\in\mathcal S$, then $\mu=\nu$ \cite[Theorem~6.13]{fathi-book}. The second is that, by Thurston's Compactification Theorem \cite[Section~8]{fathi-book}, the space $\mathcal{ML}(\Sigma)$ is homeomorphic to $\R^{6(g-1)}$. The third is that the set $\{t\delta_\gamma:\,t>0, \gamma\in\mathcal S\}$ is dense in $\mathcal{ML}(\Sigma)$ (see Proposition~6.18 in \cite{fathi-book}).
	
	In his study of $\mathcal{ML}(\Sigma)$, Thurston \cite[Chapters $8$ and $9$]{thurston} showed that $\mathcal{ML}(\Sigma)$ has a natural piecewise linear integral structure. Hence $\mathcal{ML}(\Sigma)$ admits charts given by convex cones with finitely many faces. For our purposes we need the following: 
	
	There are two convex open cones $V, V^*\subset \R^{6(g-1)}$ and homeomorphisms onto their images $T:V\rightarrow \mathcal{ML}(\Sigma)$, $T^*:V^*\rightarrow \mathcal{ML}(\Sigma)$ which commute with scaling and such that $i(T(x),T^*(y))=x\cdot y>0$.  This corresponds to Proposition 9.7.4 in \cite{thurston} or, for a rigorous treatment, Section $3$ in \cite{penner} (more specifically, \cite[page $197$]{penner}).

	The following facts will be used without further justification. If $\nu\in \mathcal C(\Sigma)$ is such that every complete geodesic in $\hat \Sigma$ intersects transversely some geodesic in the support of $\nu$,  the set
	\begin{equation}\label{compactness}
		\{\tau\in\mathcal C(\Sigma): i(\tau,\nu)\leq K\}
	\end{equation}
	is compact in $\mathcal C(\Sigma)$ (Proposition~4 in \cite{bonahon2}). Thus the map $\rho\in\mathcal T(\Sigma)\mapsto i(\lambda_\sigma,\lambda_\rho)$ is proper.
	
	If $\{\mu_i\}_{i\in\N}$ is a sequence in $\mathcal C(\Sigma)$ such that for every simple closed geodesic $\gamma$ we have
	\[
	\sup_{i\in\N} i(\mu_i,\delta_{\gamma})<\infty,
	\]
	then one can extract a convergent subsequence. Indeed, one may choose finitely many simple closed geodesics $\gamma_1,\ldots,\gamma_N$ with the property that every closed geodesic intersects transversely at least one of them. The conclusion then follows by applying \eqref{compactness} with
	\[
	\nu=\sum_{j=1}^N \delta_{\gamma_j}.
	\]
	
	If $\{\sigma_i\}_{i\in\N}$ is a divergent sequence in $\mathcal T(\Sigma)$ and $\{t_i\}_{i\in\N}$ is a sequence of positive scalars such that $\lambda_{\sigma_i}/t_i\to \mu\in \mathcal C(\Sigma)$, then $\mu\in \mathcal{ML}(\Sigma)$. The reason is that we must have $t_i\to\infty$ and $i(\lambda_{\sigma_i},\lambda_{\sigma_i})=\pi^2|\chi(\Sigma)|$.

\subsection{The mapping class group}\label{mpg.subsection}
	The group $\text{Map}(\Sigma)$ contains several important elements. {Given a simple closed geodesic $\alpha$,   $D_\alpha\in \text{Map}(\Sigma)$ denotes a Dehn twist about $\alpha$ and is defined as the identity outside a tubular neighborhood of $\alpha$ and as a full (left) twist inside the  tubular neighborhood of $\alpha$ (see \cite[Section~3.1.1]{farb-margalit}). The element $D_\alpha\in \text{Map}(\Sigma)$ depends only on $[\alpha]$ and on the orientation chosen for $\Sigma$.}
	
	We argue next that the Dehn twist $D_\alpha\in \text{Map}(\Sigma)$ satisfies 
	\begin{equation}\label{dehn.twist.limit}
		\lim_{n\to \infty}\frac{D^n_\alpha(\tau)}{n}=i(\tau,\delta_\alpha)\delta_\alpha.
	\end{equation}
	This follows from the proof of Proposition A.1 in \cite{fathi-book}, where the following inequality
	$$|i(D^n_\alpha(\delta_\gamma),\delta_\beta)-ni(\delta_\alpha,\delta_\gamma)i(\delta_\alpha,\delta_\beta)|\leq i(\delta_\gamma,\delta_\beta), \quad n\in\N$$
	is proved for every pair of simple closed geodesics $\gamma$ and $\beta$. An inspection of the proof shows that $\gamma$ does not need to be simple (the argument is local and takes place in a small tubular neighborhood of $\alpha$). Using \eqref{prop2.bonahon} we obtain that for all $\tau \in \mathcal{C}(\Sigma)$,
	$$|i(D^n_\alpha(\tau),\delta_\beta)-ni(\delta_\alpha,\tau)i(\delta_\alpha,\delta_\beta)|\leq i(\tau,\delta_\beta), \quad n\in\N.$$
	
	Thus,  for every simple closed geodesic $\beta$, 	
	\begin{equation}\label{dehn.twist.limit2}
		\lim_{n\to\infty}i(D^n_\alpha(\tau),\delta_\beta)/n=i(\delta_\alpha,\tau)i(\delta_\alpha,\delta_\beta).\end{equation}
	Every convergent subsequence of $\{D^n_\alpha(\tau)/n\}_n$  must converge to an  element of $\mathcal{ML}(\Sigma)$ because its self-intersection is $i(\tau,\tau)/n^2$.
	Combined with  \eqref{dehn.twist.limit2}, this implies  \eqref{dehn.twist.limit}.

	Another important class of elements in $\text{Map}(\Sigma)$ is given by pseudo-Anosov diffeomorphisms. For each such map $\phi\in \text{Map}(\Sigma)$ we can find two elements $\tau_+,\tau_-\in \mathcal{ML}(\Sigma)$ so that $i(\tau_-,\tau_+)>0$ and a constant $c>1$ such that
	$$\phi(\tau_-)=c^{-1}\tau_-\quad\mbox{and}\quad \phi(\tau_+)=c\tau_+.$$
	They can be seen as analogues of conformal maps in the unit ball, with a ``south pole'' and ``north pole''. We need two properties. The first  was stated by Thurston in \cite{thurston3}, and a proof can be found in \cite[Lemma~A.4]{ivanov}. It says that if $\lambda\in \mathcal{ML}(\Sigma)$ is not a multiple of either $\tau_-$ or $\tau_+$ then
	\begin{equation}\label{pseudo.anosov}
		\lim_{n\to\infty}\frac{\phi^n(\lambda)}{c^n}=\frac{i(\tau_-,\lambda)}{i(\tau_-,\tau_+)}\tau_+\quad\mbox{and}\quad\lim_{n\to\infty}\frac{\phi^{-n}(\lambda)}{c^n}=\frac{i(\tau_+,\lambda)}{i(\tau_-,\tau_+)}	\tau_-.
	\end{equation}
	The second property says that the set of pairs $(\tau_-,\tau_+)\in \mathcal{ML}(\Sigma)\times \mathcal{ML}(\Sigma)$ which arise as the ``south pole'' and ``north pole'' for some pseudo-Anosov diffeomorphism is dense in $\mathcal{ML}(\Sigma)\times \mathcal{ML}(\Sigma)$. This can be found in Proposition $4.6$ of \cite{hamenstadt}.

\subsection{Harmonic maps and Measured foliations} The references for this section are   \cite{minsky}  and \cite{wolf1}.

{Consider $\tau$ a smooth metric on $\Sigma$ and  $\rho \in \mathcal M_{\leq 0}(\Sigma)$. Given  a smooth map $w:\Sigma\rightarrow \Sigma$, the energy is defined as 
$$E_{\tau,\rho}(w)=\frac 1 2\int_\Sigma|dw|^2_{\rho}dA_\tau,$$
where
$$
	|dw|^2_\rho(p)={\rho(dw_{p}(e_1),dw_{p}(e_1))+\rho(dw_{p}(e_2),dw_{p}(e_2))}
	$$
	where $\{e_1,e_2\}$ is a $\tau$-orthonormal basis of $T_p\Sigma$.

There is a unique harmonic diffeomorphism $w:(\Sigma,\tau)\rightarrow (\Sigma,\rho)$ homotopic to the identity  \cite[Corollary, page 271]{schoen-yau-univalent}. This induces a   holomorphic quadratic differential $\Phi(\rho)$, called {\em Hopf differential}, on $\Sigma$ which corresponds to the $(2,0)$-part of $w^*\rho$ with respect to the complex structure induced by $\tau$.  With respect to local complex coordinates $z=x+iy$ induced by $\tau$ on an open set of  $\Sigma$ 
$$
\Phi(\rho)=\Phi(z)dz^2=\left(\frac 1 4(|\partial_x w|_\rho^2-|\partial_y w|_\rho^2)-\frac i 2\rho(\partial_x w,\partial_y w)\right)dz^2.
$$
When there is no risk of confusion we denote $\Phi(\rho)$ by $\Phi$.

There are several identities relating the pointwise energy of $w$, its jacobian, and $\Phi$. For our purposes we need (see $(3.4)$ in \cite{minsky})
\begin{equation}\label{ejphi}
E_{\tau,\rho}(w)\leq 2||\Phi||+\area_\rho(\Sigma)\quad \text{where }||\Phi||=\int_\Sigma|\Phi|dA_\tau.
\end{equation}

Similar to the space of measured geodesic laminations $\mathcal{ML}(\Sigma)$ there is the space of measured foliations $\mathcal{MF}(\Sigma)$, whose objects $\mu=(\mathcal F,|dv|)$ are a foliation $\mathcal F$ of $\Sigma$ (possibly singular at a finite set of points) and a transverse measure $|dv|$. The reader can consult \cite[Section 11.3]{kapovich} for the precise definitions. There is a concept of ``straightening'' the leaves which gives a homeomorphism between  $\mathcal {MF}(\Sigma)$ and $\mathcal {ML}(\Sigma)$ \cite[11.8-11.9]{kapovich}. Under this homeomorphism, for every closed geodesic $\gamma$, (see \cite[Chapter~9, Prop. 2.21]{martelli})
\begin{equation}\label{foliation.lamination}
i(\mu,\delta_\gamma)=\inf_{c\in [\gamma]}\int_{S^1}|dv(c'(\theta))|d\theta.
\end{equation}

Away from the {finitely many} zeroes of $\Phi=\Phi(\rho)\neq 0$ we can find complex coordinates $z=x+iy$ (called {\em flat coordinates}) so that $\Phi=dz^2=dx^2-dy^2+2idxdy$. The curves with $x$ constant form a vertical foliation which extends to a foliation of $\Sigma$ that is singular at the zeroes of $\Phi$.  The  transverse measure $|dx|$ makes this a measured foliation denoted by $\mu_V(\Phi)$. Likewise, the curves with $y$ constant form a horizontal foliation and the transverse measure $|dy|$ makes it a measured foliation $\mu_H(\Phi)$.

The intersection between  the vertical and horizontal foliations recovers the norm of $\Phi$. More precisely, (see \cite[p.~22]{mirzakhani-earthquake})
\begin{equation}\label{vert.hor.inter}
i(\mu_V(\Phi),\mu_H(\Phi))=||\Phi||.
\end{equation}
The next lemma  is proven in  \cite[Lemma 4.6]{wolf1}.  We include its proof because a different normalization of $\mu_V(\Phi)$ is used. 
\begin{lemm}\label{wolf.lemma4.6} For every $\nu\in \mathcal C(\Sigma)$ we have
$$2i(\mu_V(\Phi),\nu)\leq i(\lambda_\rho,\nu).$$
\end{lemm}
\begin{proof}
Consider $c$ a curve homotopic to a closed {$\tau$-}geodesic $\gamma$ so that $w(c)$ is a $\rho$-geodesic homotopic to $\gamma$. Using flat coordinates $z=x+iy$ under which $\Phi=dz^2$ we have from \cite[(3.3)]{minsky} that
$w^*\rho\geq 4dx^2.$
Thus $$|c'(\theta)|_{w^*\rho}\geq 2|dx(c'(\theta))|\quad\text{if}\quad \Phi_{|c(\theta)}\neq 0.$$   We obtain from \eqref{foliation.lamination} that
$$i(\lambda_\rho,\delta_\gamma)=l_{w^*\rho}(c)=\int_{S^1}|c'(\theta)|_{w^*\rho}d\theta\geq 2 \int_{S^1}|dx(c'(\theta))|d\theta\geq2i(\mu_V(\Phi),\delta_\gamma).$$
The desired result follows from \eqref{prop2.bonahon}.
\end{proof}}

\subsection{Minimal surfaces  in $\Sigma\times \Sigma$}\label{top.argument}
%Given $\Pi\in  \mathcal L^*(g)$, it follows from the proof of \cite[Theorem 1]{lee} that there is $f\in \text{Diff}(\Sigma)$ so that $ \text{area}_{h}(\Pi)= \text{area}_{h}(\text{graph}(f))$.

%Y. I. Lee showed in \cite[Theorem 1]{lee}\footnote{The existence part only needs the metrics to have nonpositive curvature.} the existence of a minimal diffeomorphism  with the property that its graph minimizes area in its homotopy class. More precisely, let $\Pi_0$ denote the  homotopy class of the diagonal map $x\in\Sigma\mapsto(x,x)\in \Sigma\times \Sigma$. There is $f\in \text{Diff}(\Sigma)$ so that $$A(\sigma,\rho)=\text{area}_{\sigma\oplus\rho}(\text{graph}(f))=\inf\{\text{area}_{\sigma\oplus\rho}(S):S\in \Pi_0\}.$$

% Say that  \iota:T^2\rightarrow \Sigma\times \Sigma is a  pi_1 injective harmonic map, Sigma has genus\geq 2,  and we consider a product metric \sigma\oplus\rho on \Sigma\times \Sigma, where \rho and \sigma are nonpositively curved. Why are there closed geodesics alpha (for sigma) and beta (for rho) so that iota(T^2)=\alpha\times\beta? is it because if \iota_1:T^2\to (\Sigma,\sigma)is a harmonic map then the image is a closed geodesic? does the curvature of sigma or rho play a role?

{We consider a   nonpositively curved metric $h$ on $\Sigma\times\Sigma$ which must be isometric to $\sigma\oplus\rho$, where $\rho, \sigma\in \mathcal M_{\leq 0}(\Sigma)$. }

	{Let $\mathcal C^*$ denote the set of  primitive conjugacy classes of  closed curves in $\Sigma$, where each class is identified with its inverse.}
	\begin{lemm}\label{top.argument.genus1}
	For every $\Pi\in \mathcal L^*(1)$ there  are unique $$(\alpha(\Pi),\beta(\Pi))=(\alpha,\beta)\in \mathcal C^*\times  \mathcal C^*$$ and $d(\Pi)=d\in\N$ so that $\text{area}_{h}(\Pi)=dl_\sigma(\alpha)l_\rho(\beta)$. Moreover, given $d\in\N$ and $(\alpha,\beta)\in \mathcal C^*\times  \mathcal C^* $ we have
		$$\#\{\Pi\in \mathcal L^*(1): \alpha(\Pi)=\alpha, \beta(\Pi)=\beta, d(\Pi)=d\}\leq d^2.$$
	\end{lemm}
	\begin{proof}
	{Schoen--Yau \cite{schoen-yau} gives $(T^2,\iota=(\iota_1,\iota_2))\in \Pi$, an area-minimizing branched immersion achieving  $\area_h(\Pi)$. The maps $\iota_1$ and $\iota_2$ are harmonic and so we obtain from \cite{eells-sampson} that each map is totally geodesic with image a point or a closed geodesic. The map $\iota$ is $\pi_1$-injective and so the latter holds for $\iota_1$ and $\iota_2$.  Thus we find  a unique pair $(\alpha(\Pi),\beta(\Pi))\in  \mathcal C^*\times  \mathcal C^*$ so that $\iota_1(T^2)$ and $\iota_2(T^2)$ represent $\alpha(\Pi)$ and $\beta(\Pi)$, respectively, and 
$$\iota(T^2)=\iota_1(T^2)\times\iota_2(T^2).$$ 
The group $\iota_*(\pi_1(S^1\times S^1))$ has finite index in $\pi_1(\iota_1(T^2)\times\iota_2(T^2))$, and so $\iota$ is a degree $d=d(\Pi)$ covering map. Thus $\text{area}_{\sigma\oplus\rho}(\Pi)=dl_\sigma(\alpha)l_\rho(\beta)$.
		The rest of the proof follows as in  Lemma \ref{torus.cover}.}
	\end{proof}

	Consider a surface $S$ of genus $k>1$. Fix a finite cover {$p_i$} from $S$ to $\Sigma$ in each equivalence class of $C_g(k)${, so that $C_g(k)=\{[p_i]\}_{i=1}^n$.  Recall that two covers are equivalent if, up to precomposing with an element of the mapping class group of $S$, they are homotopic.} Given $p_1,p_2\in  C_g(k)$ and $\phi\in \text{Map}(S)$, after choosing a representative of $\phi$ we denote by $\Pi(p_1,p_2,\phi)\in \mathcal L^*(k)$ the conjugacy class induced by the immersion
	$$\iota:S\rightarrow \Sigma\times\Sigma,\quad\iota(x)=(p_1(\phi(x)),p_2(x)).$$
	If $S=\Sigma$, we simplify the notation to $\Pi(\phi)\in \mathcal L^*(g)$, $\phi\in \text{Map}(\Sigma)$.
	
	\begin{lemm}\label{top.argument2} Consider $k>1$. If $k-1$ is not a positive multiple of $g-1$ then $\mathcal L^*(k)=\emptyset$. If $k-1$ is a positive multiple of $g-1$, the assignment 
		$$(p_1,p_2,\phi)\in C_g(k)\times C_g(k)\times  \text{Map}(S)\mapsto  \Pi(p_1,p_2,\phi)\in \mathcal L^*(k)$$
		is surjective. Moreover, $\Pi(p_1,p_2,\phi)=\Pi(q_1,q_2,\psi)$ if and only if $p_1=q_1$, $p_2=q_2$, and
		$\psi=T_1\circ \phi\circ T_2$
		in $\text{Map}(S)$ for some $T_1\in D(p_1)$ and $T_2\in D(p_2)$.
	\end{lemm}
	\begin{proof}
		Consider $\Pi\in \mathcal L^*(k)$. It is induced  by a  $\pi_1$-injective immersion 
		$$
		\iota=(\iota_1,\iota_2):S\rightarrow \Sigma\times \Sigma\quad\text{with $S$ a surface of genus $k$ and }\int_S\iota^*\bar \omega=0.
		$$
 Suppose that neither $(\iota_i)_*$, $i=1,2$, is injective. In this case, there exist elements $\alpha_i\in \pi_1(S)$ such that $(\iota_i)_{*}(\alpha_i)=0$, $i=1,2$. If $\alpha_1$ and $\alpha_2$ commute, then $\alpha_1^{m_1}=\alpha_2^{m_2}$ for some $m_1,m_2\in\Z\setminus\{0\}$. Thus $\iota_*(\alpha_1^{m_1})$ is trivial in $\pi_1(\Sigma\times\Sigma)$. If $\alpha_1$ and $\alpha_2$ do not commute, then $\gamma=\alpha_1\alpha_2\alpha_1^{-1}\alpha_2^{-1}$ is such that $\iota_*(\gamma)$ is trivial in $\pi_1(\Sigma\times\Sigma)$. Both cases contradict $\iota$ being $\pi_1$-injective.

Thus, without loss of generality, $(\iota_1)_*$ is injective. In this case we know from \cite{scott} that $(\iota_1)_*(\pi_1(S))$ has finite index in $\pi_1(\Sigma)$. Therefore $\iota_1$ is homotopic to a covering map of degree $d>0$. As a result, $\chi(S)=d\chi(\Sigma)$. In particular, $k-1 =d(g-1).$ The condition $\int_S\iota^*\bar\omega=0$ implies that $\iota_2$ also has degree $d$. {From the equality case in Kneser’s inequality \cite{kneser}, $\iota_2$ is also  homotopic to a covering map of degree $d>0$.

Thus $\iota$ is homotopic to an immersion  $\iota'=(\iota'_1,\iota'_2):S\rightarrow \Sigma\times\Sigma$, where $\iota'_1$ and $\iota'_2$ are both covering maps of degree $d$.}  There is a unique $p_2\in C_g(k)$ such that, after precomposing $\iota'$ with some $T\in  \text{Diff}_+(S)$ (and still denoting the new map by $\iota'$), we have that $\iota'_2$ is homotopic to $p_2$. There is a unique $p_1\in C_g(k)$ and some $\phi\in  \text{Diff}_+(S)$ such that $\iota'_1$ is homotopic to $p_1\circ\phi$. Hence $\Pi=\Pi(p_1,p_2,\phi)$. This shows that the assignment is surjective.
		
		Suppose now that $\Pi(p_1,p_2,\phi)=\Pi(q_1,q_2,\psi)$, and choose representatives $\phi,\psi\in \text{Diff}_+(S)$. Arguing as in Lemma \ref{conjugacy.one}, there exists $T\in \text{Diff}_+(S)$ such that $(q_1\circ\psi\circ T,q_2\circ T)$ is homotopic to $(p_1\circ\phi,p_2)$. Hence $q_2\circ T$ is homotopic to $p_2$, so by the choice of the representatives of $C_g(k)$ we have $q_2=p_2$ and $T\in D(p_2)$. Also, $q_1\circ\psi\circ T$ is homotopic to $p_1\circ\phi$, and therefore $q_1\circ\psi\circ T\circ\phi^{-1}$ is homotopic to $p_1$. Again by the choice of the representatives of $C_g(k)$, we have $q_1=p_1$ and $\psi\circ T\circ\phi^{-1}\in D(p_1)$. Thus
		$$\psi=T_1\circ \phi\circ T_2$$
		in $\text{Map}(S)$ for some $T_1\in D(p_1)$ and $T_2\in D(p_2)$.
		
		Conversely, if $\psi=T_1\circ \phi\circ T_2$ in $\text{Map}(S)$ with $T_1\in D(p_1)$ and $T_2\in D(p_2)$, choose representatives so that this equality holds in $\text{Diff}_+(S)$. Then
		$$(p_1\circ\psi,p_2)=(p_1\circ\phi\circ T_2,p_2)=(p_1\circ\phi,p_2)\circ T_2,$$
		and therefore $\Pi(p_1,p_2,\psi)=\Pi(p_1,p_2,\phi)$.
	\end{proof}	
	
	{The previous lemma has the following consequence. Consider  $$\Pi(p_1,p_2,\phi)\in \mathcal L^*(k)\quad\text{with $k=d(g-1)+1$ and }d\in \N.$$  We have from \cite{schoen-yau} the existence of $(S,\iota)\in \Pi$  so that
$$\iota=(\iota_1,\iota_2):S\rightarrow \Sigma\times\Sigma $$		
is an area-minimizing branched immersion  achieving $\area_h(\Pi)$. The maps   $\iota_1, \iota_2$ are harmonic, homotopic to  covering maps of degree $d$ if $d\geq 1$ and thus covering maps of degree $d$ (see \cite[Section 3]{schoen-yau-univalent}). Thus we find $T\in \text{Diff}_0(S)$ and $\psi \in  \text{Diff}_+(S)$ homotopic to $\phi$ so that $\iota\circ T=(p_1\circ \psi,p_2)$.  As a result, the minimal immersion is unique (see \cite{markovic}). 
When $h$ is K\"{a}hler--Einstein, Y. I. Lee showed in \cite[Theorem 1 and Theorem 2]{lee} that the minimal immersion $\iota$ is Lagrangian and thus the  assignment $L\in \mathcal L(k)\mapsto [\pi_1(L)]\in \mathcal L^*(k)$ is a bijection.}

\section{Proof of Theorems in $\Sigma\times \Sigma$}	
\subsection{Proof of Theorem \ref{lagrangian.tori2}} We start by recalling the statement.
	\begin{thm}
				Consider a nonpositively curved metric $h=\sigma\oplus\rho$ on $\Sigma\times \Sigma$. We have
{$$\lim_{A\to\infty}\frac{\ln(\#\{\Pi\in \mathcal L^*(1):\text{area}_h(\Pi)\leq A\})}{A}=\max\left\{\frac{\delta(\rho)}{\text{sys}(\sigma)},\frac{\delta(\sigma)}{\text{sys}(\rho)}\right\}.$$}
	\end{thm}
	\begin{proof}    
		Assume without loss of generality {that $\frac{\delta(\rho)}{\text{sys}(\sigma)}\geq \frac{\delta(\sigma)}{\text{sys}(\rho)}$.
		
		Let $\alpha(\Pi),\beta(\Pi)$ and $d(\Pi)$ be as in Lemma \ref{top.argument.genus1}. There is $c_0=c_0(\sigma,\rho)$ so that if $d\,l_\sigma(\alpha)l_\rho(\beta)\leq A$ then $d<c_0A$. If we set $N=\lfloor c_0A \rfloor$, we have from Lemma \ref{top.argument.genus1} that
		\begin{multline*}
		\{\Pi\in \mathcal L^*(1):\text{area}_{\sigma\oplus\rho}(\Pi)\leq A\}\\
		\subset \bigcup_{d=1}^N\{\Pi\in \mathcal L^*(1):l_\sigma(\alpha(\Pi))l_\rho(\beta(\Pi))\leq A, d(\Pi)=d\}.
		\end{multline*}
		Using Lemma \ref{top.argument.genus1} again we see that
		\begin{multline*}
		\#\{\Pi\in \mathcal L^*(1):\text{area}_{\sigma\oplus\rho}(\Pi)\leq A\}\\
		\leq c_0^3A^3\#\{(\alpha,\beta)\in \mathcal C^*\times \mathcal C^*:l_\sigma(\alpha)l_\rho(\beta)\leq A\}.
		\end{multline*}
		Using the fact that $(\Sigma,\sigma)$ and $(\Sigma,\rho)$ have, respectively, $e^{\delta(\sigma)L+o(L)}$ and $e^{\delta(\rho)L+o(L)}$ elements of $\mathcal C^*$ of length less than or equal to $L$, we have
		\begin{multline*}
			\#\{(\alpha,\beta)\in \mathcal C^*\times \mathcal C^*:l_\sigma(\alpha)l_\rho(\beta)\leq A\}
			\\
			\leq \#\{(\alpha,\beta)\in \mathcal C^*\times \mathcal C^*: l_\sigma(\alpha)\leq \sqrt A, l_\rho(\beta)\leq A/\text{sys}(\sigma)\}\\
			\quad+\#\{(\alpha,\beta)\in \mathcal C^*\times \mathcal C^*: l_\sigma(\alpha)\leq A/\text{sys}(\rho), l_\rho(\beta)\leq \sqrt A\}\\
			\leq 2e^{o(A)}e^{\delta(\rho)A/\text{sys}(\sigma)}=e^{\delta(\rho)A/\text{sys}(\sigma)+o(A)}.
		\end{multline*}
		Hence we obtain that
		\begin{multline*}
			\limsup_{A\to \infty}\frac{\ln \#\{\Pi\in \mathcal L^*(1):\text{area}_{\sigma\oplus\rho}(\Pi)\leq A\}}{A}\\
			\leq \limsup_{A\to \infty}\frac{\ln \#\{(\alpha,\beta)\in \mathcal C^*\times \mathcal C^*:l_\sigma(\alpha)l_\rho(\beta)\leq A\}}{A}\leq \frac{\delta(\rho)}{\text{sys}(\sigma)}.
		\end{multline*}
		The lower bound follows in the same way because if $\alpha_0\in \mathcal C^*$ satisfies $l_\sigma(\alpha_0)=\text{sys}(\sigma)$, then for every $\beta\in \mathcal C^*$ the product of geodesics representing $\alpha_0$ and $\beta$ produces an element  $\Pi\in \mathcal L^*(1)$ and
		$$\text{area}_{\sigma\oplus\rho}(\Pi)=\text{sys}(\sigma)l_\rho(\beta).$$
		Therefore
		$$\#\{\beta\in \mathcal C^*:\text{sys}(\sigma)l_\rho(\beta)\leq A\}\leq
		 \#\{\Pi\in \mathcal L^*(1):\text{area}_{\sigma\oplus\rho}(\Pi)\leq A\}.$$
		Using 
		 \cite{margulis}\footnote{More precisely, the extension to nonpositive curvature proven in \cite{knieper}.} in the last inequality we obtain
		\begin{multline*}
			\liminf_{A\to \infty}\frac{\ln \#\{\Pi\in \mathcal L^*(1):\text{area}_{\sigma\oplus\rho}(\Pi)\leq A\}}{A}\\
			\geq \lim_{A\to \infty}\frac{\ln \#\{\beta\in \mathcal C^*:\text{sys}(\sigma)l_\rho(\beta)\leq A\}}{A}=\frac{\delta(\rho)}{\text{sys}(\sigma)}.
		\end{multline*}
		}

	\end{proof}
	
\subsection{Proof of Theorem \ref{main.thm.variable}}

	We denote $\Pi(Id)\in \mathcal L^*(g)$ by $\Pi_0$. Given $\sigma,\rho$  metrics  on $\Sigma$  set
	$$A(\sigma,\rho)=\text{area}_{\sigma\oplus\rho}(\Pi_0).$$
	This induces a function on $\mathcal T(\Sigma)\times\mathcal T(\Sigma)$ given by 
	$$A: \mathcal T(\Sigma)\times\mathcal T(\Sigma)\rightarrow (0,+\infty),\quad (\sigma,\rho)\mapsto A(\sigma,\rho).$$
	
The next proposition is the analogue of Proposition \ref{area.estimate} but for $\Sigma\times \Sigma$.
		\begin{prop}\label{prop2}
			{We have for all $\sigma,\rho\in \mathcal{M}_{\leq 0}(\Sigma)$ that
			$$i(\lambda_\sigma,\lambda_\rho)\leq A(\sigma,\rho)\leq i(\lambda_\sigma,\lambda_\rho)+
			\area_\rho(\Sigma)+\area_\sigma(\Sigma).$$}
		\end{prop}
		\begin{proof}
{Consider the minimal immersion $(\Sigma,\iota)\in \Pi_0$
$$\iota=(\iota_1,\iota_2):\Sigma\rightarrow \Sigma\times\Sigma $$		
 with area $A(\sigma,\rho)$. Set  $\tau=\iota^*(\sigma\oplus\rho)$ a smooth metric on $\Sigma$. Then $\iota$ is a conformal harmonic map and we have
$$A(\sigma,\rho)=\frac{1}{2}\int_\Sigma|d\iota|_{\sigma\oplus\rho}^2dA_\tau=E_{\tau,\sigma}(\iota_1)+E_{\tau,\rho}(\iota_2).$$
Both $\iota_1$ and $\iota_2$ are harmonic maps and the fact that $\iota$ is conformal implies that $\Psi=\Phi(\rho)=-\Phi(\sigma)$. Using \eqref{ejphi} and then   \eqref{vert.hor.inter} we deduce
\begin{align*}
A(\sigma,\rho)&\leq 4||\Psi||+\area_\sigma(\Sigma)+\area_\rho(\Sigma)\\
&=4i(\mu_V(-\Psi),\mu_H(-\Psi))+\area_\sigma(\Sigma)+\area_\rho(\Sigma).
\end{align*}
We have $\mu_H(-\Psi)=\mu_V(\Phi(\rho)).$ Using	 Lemma \ref{wolf.lemma4.6} with $\nu=\mu_V(-\Psi)$ we obtain
	$$A(\sigma,\rho)\leq 2i(\mu_V(-\Psi),\lambda_\rho)+\area_\sigma(\Sigma)+\area_\rho(\Sigma).$$
	Using  Lemma \ref{wolf.lemma4.6} again with $\nu=\lambda_\rho$ and $\Phi=\Phi(\sigma)=-\Psi$ we obtain
	$$A(\sigma,\rho)\leq i(\lambda_\sigma,\lambda_\rho)+\area_\sigma(\Sigma)+\area_\rho(\Sigma).$$	
	This proves the upper bound.} Next we prove the lower bound and this concludes the proof.
\begin{lemm}\label{left.hand.side}
$A(\sigma,\rho)\geq i(\lambda_\sigma,\lambda_\rho).$
			\end{lemm}
			\begin{proof}
				{The maps $\iota_1,\iota_2$ are in $\text{Diff}_+(\Sigma)_0 $ and so $\phi=\iota_2\circ\iota_1^{-1}$ is such that}  
				$$\iota(\phi):\Sigma\rightarrow \Sigma\times\Sigma,\quad x\mapsto (x,\phi(x))$$
				is also a minimal immersion with area  $A(\sigma,\rho)$.
				
Use $d\phi_p^{\top}$ to denote the adjoint of the map $d\phi_p:(T_p\Sigma,\sigma)\rightarrow (T_{\phi(p)}\Sigma,\rho).$ 
				Let $\lambda^2_1\geq \lambda_2^2> 0$ denote the eigenvalues of $d\phi^{\top}_p\circ d\phi_p$. We have $|d\phi|_\rho^2(p)= \lambda_1^2+ \lambda_2^2$ and $|\text{Jac}(\phi)|_\rho^2(p)=\lambda_1^2\lambda_2^2$.
				Thus
				\begin{align}\label{1.bound}
					A(\sigma,\rho)&=\int_{\Sigma}\sqrt{1+|d\phi|_\rho^2+|\text{Jac}(\phi)|_\rho^2}dA _\sigma
					\geq
					\int_{\Sigma}\sqrt{(|\lambda_1|+|\lambda_2|)^2}dA_\sigma \\ \notag
					&=   \int_{\Sigma}{|\lambda_1|+|\lambda_2|}dA_\sigma .
				\end{align}
				With $dV_\sigma$ being the Liouville measure on the unit tangent bundle $T^1\Sigma$ we have, using \eqref{1.bound} on the last inequality,
				\begin{align*}
					\int_{T^1\Sigma}|d\phi(v)|_\rho dV_{\sigma}&=\int_{\Sigma}\int_0^{2\pi}(\lambda_1^2\cos^2(\theta)+\lambda_2^2\sin^2(\theta))^{1/2}d\theta dA_\sigma\\ \notag
					&\leq \int_{\Sigma}\int_0^{2\pi}|\lambda_1||\cos(\theta)|+|\lambda_2||\sin(\theta)|d\theta dA_\sigma\\ \notag
					&\leq 4\int_{\Sigma}|\lambda_1|+|\lambda_2|dA_\sigma\leq 4A(\sigma,\rho).
				\end{align*}
				Hence we deduce
				\begin{equation}\label{area.bound.below}
					A(\sigma,\rho)\geq \frac{1}{4}\int_{T^1\Sigma}|d\phi(v)|_\rho dV_{\sigma}.
				\end{equation}
				Croke--Fathi's proof of  \cite[Theorem~2.1]{croke} shows\footnote{The factor of $4$ comes from a different normalization. We have $i(\lambda_\sigma,\lambda_\sigma)=\frac{\pi}{2}\area_\sigma(\Sigma).$}  that $$\int_{T^1\Sigma}|d\phi(v)|_\rho dV_{\sigma}\geq 4i(\lambda_\sigma,\lambda_\rho).$$
				
				Combining this inequality with \eqref{area.bound.below} we see that $A(\sigma,\rho)\geq i(\lambda_\sigma,\lambda_\rho)$.
			\end{proof}
		\end{proof}
\begin{thm}
		Consider a nonpositively curved metric $h=\sigma\oplus\rho$ on $\Sigma\times \Sigma$. We have
		\begin{equation}\label{genusg.couting}
		\lim_{A\to\infty}\frac{\#\{\Pi\in \mathcal L^*(g):\text{area}_h(\Pi)\leq A\}}{A^{6(g-1)}}=\frac{B(\sigma)B(\rho)}{b_g}.
		\end{equation}
		
		If $k-1\geq 1$ is not a positive integral multiple of $g-1$ then $\mathcal L^*(k)=\emptyset$.
		
		If $k-1$ is a positive multiple of $g-1$ then
		$$\lim_{A\to\infty}\frac{\#\{\Pi\in \mathcal L^*(k):\text{area}_h(\Pi)\leq A\}}{A^{6(k-1)}}=\sum_{p_1,p_2\in C_g(k)}\frac{B(p_1^*\sigma)B(p_2^*\rho)}{|D(p_1)||D(p_2)|b_k}.$$
	\end{thm}
	\begin{proof}
		If $k-1$ is not a positive multiple of $g-1$ we have from Lemma \ref{top.argument2} that $\mathcal L^*(k)=\emptyset$.

		Consider the homogeneous function
		$$\Lambda:\mathcal C(\Sigma)\rightarrow [0,\infty), \quad \Lambda(\lambda)=l_\sigma(\lambda)=i(\lambda_{\sigma},\lambda).$$
		Rafi--Souto \cite{souto-rafi} further extended the work of Mirzakhani \cite{mirza} and showed 
		\begin{equation}\label{count.mirza}
			\lim_{A\to\infty}\frac{\#\{\phi\in \text{Map}(\Sigma):\Lambda(\phi(\lambda_\rho))\leq A\}}{A^{6(g-1)}}=\frac{B(\rho)B(\sigma)}{b_g}.
		\end{equation}
		We have from Proposition \ref{prop2} that
		\begin{multline*}
			\#\{\phi\in \text{Map}(\Sigma):\Lambda(\lambda_{\phi_*\rho})+\area_\sigma(\Sigma)+\area_\rho(\Sigma)
			\leq A\}\\
			\leq \#\{\phi\in \text{Map}(\Sigma):A(\sigma,\phi_*\rho)\leq A\}\leq \#\{\phi\in \text{Map}(\Sigma):\Lambda(\lambda_{\phi_*\rho})\leq A\}.
		\end{multline*}
		Thus, using the fact that   $\lambda_{\phi_*\rho}=\phi(\lambda_\rho)$, we obtain from \eqref{count.mirza} that 
		\begin{equation}\label{count.mirza2}
			\lim_{A\to\infty}\frac{\#\{\phi\in \text{Map}(\Sigma):A(\sigma,\phi_*\rho)\leq A\}}{A^{6(g-1)}}=\frac{B(\rho)B(\sigma)}{b_g}.
		\end{equation}
		From Lemma \ref{top.argument2}  we have the disjoint decomposition  
		$$ \mathcal L^*(g)=\cup_{\phi\in \text{Map}(\Sigma)}\{\Pi(\phi)\}.$$
		
		\begin{lemm}\label{area.homotopic.class} For every representative of $\phi\in \text{Map}(\Sigma)$ we have $$\text{area}_{\sigma\oplus\rho}(\Pi(\phi))=A(\sigma,\phi_*\rho).$$
		\end{lemm}
		\begin{proof}
			Consider the diffeomorphism 
			$$F:\Sigma\times\Sigma\rightarrow \Sigma\times \Sigma,\quad F(x,y)=(x,\phi^{-1}(y)).$$
			We have $F^*(\sigma\oplus\rho)=\sigma\oplus\phi_*\rho$. If   $(\Sigma,\iota)\in \Pi_0$ then $(\Sigma,F\circ\iota)\in \Pi(\phi)$ and
			$$\text{area}_{(F\circ\iota)^*(\sigma\oplus\rho)}(\Sigma)=\text{area}_{\iota^*(\sigma\oplus\phi_*\rho)}(\Sigma)\implies\text{area}_{\sigma\oplus\rho}(\Pi(\phi))\leq A(\sigma,\phi_*\rho).$$
			If   $(\Sigma,\iota)\in \Pi(\phi)$ then $(\Sigma,F^{-1}\circ\iota)\in \Pi_0$ and
			$$\text{area}_{(F^{-1}\circ\iota)^*(\sigma\oplus\phi_*\rho)}(\Sigma)=\text{area}_{\iota^*(\sigma\oplus\rho)}(\Sigma)\implies A(\sigma,\phi_*\rho)\leq \text{area}_{\sigma\oplus\rho}(\Pi(\phi)).$$
		\end{proof}

		Using Lemma \ref{area.homotopic.class} we see that
		$$\#\{\Pi\in \mathcal L^*(g):\text{area}_{\sigma\oplus\rho}(\Pi)\leq A\}=\#\{\phi\in \text{Map}(\Sigma):A(\sigma,\phi_*\rho)\leq A\}.
		$$
		This identity and \eqref{count.mirza2} imply \eqref{genusg.couting}.

		\begin{prop}\label{higher.genus.count.lagrangian} Assume that $k-1$ is a positive multiple of $g-1$. Then
			$$
			\lim_{A\to\infty}\frac{\#\{\Pi\in \mathcal L^*(k):\text{area}_{\sigma\oplus\rho}(\Pi)\leq A\}}{A^{6(k-1)}}
			=\sum_{p_1,p_2\in C_g(k)}\frac{B(p_1^*\sigma)B(p_2^*\rho)}{|D(p_1)||D(p_2)|b_k}.
			$$
		\end{prop}
		\begin{proof}
		Consider  a surface $S$ with genus $k$. In what follows we use the description of $\mathcal L^*(k)$ given in Lemma \ref{top.argument2}. For  $p_1,p_2\in C_g(k)$  let $\mathcal L^*_{p_1,p_2}(k)\subset \mathcal L^*(k)$ be the subset consisting of those elements represented by $\Pi(p_1,p_2,\phi)$, $\phi\in \text{Map}(S)$. We have the disjoint union
			$$
			\mathcal L^*(k)=\bigcup_{p_1,p_2\in C_g(k)}\mathcal L^*_{p_1,p_2}(k).
			$$
			The assignment $\phi\in  \text{Map}(S)\mapsto \Pi(p_1,p_2,\phi)\in \mathcal L^*_{p_1,p_2}(k)$ induces a bijection between $\mathcal L^*_{p_1,p_2}(k)$ and the set of $D(p_1)\times D(p_2)$-orbits in $\text{Map}(S)$, where each orbit is
			$$
			(D(p_1)\times D(p_2))\cdot\phi=\{T_1\circ \phi\circ T_2:T_1\in D(p_1),T_2\in D(p_2)\}.
			$$
			
			Fix $p_1,p_2\in C_g(k)$. Consider the local isometry between $p_1^*\sigma\oplus p_2^*\rho$ and $\sigma\oplus \rho$ given by
			$$
			F_{p_1,p_2}:S\times S\rightarrow \Sigma\times \Sigma,\quad F_{p_1,p_2}(x,y)=(p_1(x),p_2(y)).
			$$
			Choose a representative of $\phi\in \text{Map}(S)$ and set 
			$$
			\hat\iota(\phi):S\to S\times S, \quad \hat\iota(\phi)(x)=(\phi(x),x).
			$$
			Then $F_{p_1,p_2}\circ \hat\iota(\phi)$ is the immersion inducing $\Pi(p_1,p_2,\phi)$. Every $(S,\iota)\in \Pi(p_1,p_2,\phi)$ lifts to $S\times S$ and thus, if $\hat\Pi(\phi)$ denotes the homotopy class of $\hat\iota(\phi)$ in $S\times S$, we have
			$$
			\text{area}_{\sigma\oplus \rho}(\Pi(p_1,p_2,\phi))=\text{area}_{p_1^*\sigma\oplus p_2^*\rho}(\hat\Pi(\phi)).
			$$
			Given $A>0$, set
			\begin{align*}
				E_{p_1,p_2}(A)&=\{\phi\in \text{Map}(S):\text{area}_{\sigma\oplus \rho}(\Pi(p_1,p_2,\phi))\leq A\},\\
				N_{p_1,p_2}(A)&=\{\Pi\in \mathcal L^*_{p_1,p_2}(k):\text{area}_{\sigma\oplus \rho}(\Pi)\leq A\}.
			\end{align*}
			Arguing exactly as in the proof of \eqref{genusg.couting}, but now on  $S\times S$ with the metric $p_1^*\sigma\oplus p_2^*\rho$, we obtain
			\begin{align}\label{growth.multiples}
				\lim_{A\to\infty}\frac{\#E_{p_1,p_2}(A)}{A^{6(k-1)}}&=\lim_{A\to\infty}\frac{\#\{\phi\in \text{Map}(S):\text{area}_{p_1^*\sigma\oplus p_2^*\rho}(\hat\Pi(\phi))\leq A\}}{A^{6(k-1)}}\\ \notag
				&=\frac{B(p_1^*\sigma)B(p_2^*\rho)}{b_k}.
			\end{align}
			
			Lemma \ref{top.argument2} implies that $\#N_{p_1,p_2}(A)$ is equal to the number of $D(p_1)\times D(p_2)$-orbits in $E_{p_1,p_2}(A)$. By Burnside's lemma,
			\[
			\#N_{p_1,p_2}(A)=\frac{1}{|D(p_1)||D(p_2)|}\sum_{(T_1,T_2)\in D(p_1)\times D(p_2)}\#E_{p_1,p_2}^{T_1,T_2}(A),
			\]
			where
			\[
			E_{p_1,p_2}^{T_1,T_2}(A)=\{\phi\in E_{p_1,p_2}(A):T_1\circ \phi\circ T_2=\phi\text{ in }\text{Map}(S)\}.
			\]
			
			We have $E_{p_1,p_2}^{Id,Id}(A)=E_{p_1,p_2}(A)$. Suppose now that $(T_1,T_2)\neq (Id,Id)$. The idea is to show that
			$$
			\#E_{p_1,p_2}^{T_1,T_2}(A)=o(A^{6(k-1)}).
			$$
			If $E_{p_1,p_2}^{T_1,T_2}(A)=\emptyset$ for every $A$, there is nothing to prove. Otherwise choose $\phi_0\in \text{Map}(S)$ such that
			$
			T_1\circ \phi_0\circ T_2=\phi_0
			$
			in $\text{Map}(S)$. Then
			$
			T_1=\phi_0\circ T_2^{-1}\circ \phi_0^{-1}
			$
			in $\text{Map}(S)$. In particular, $T_1$ and $T_2$ have the same order $m>1$. 
			
			Consider $W_i=S/\langle T_i\rangle$,  which has genus $r<k$, and  let $w_i:S\rightarrow W_i$ be the regular  cover of degree $m$. We have $\langle T_i\rangle < D(p_i)$ and so we find $u_i:W_i\rightarrow \Sigma$ finite covers with $p_i=u_i\circ w_i$, $i=1,2$.
			
			\begin{lemm}\label{cover.lemma.lagrangian} 
				There are finite covers $q_1,q_2: W_2\rightarrow \Sigma$ so that for every $\phi \in E_{p_1,p_2}^{T_1,T_2}(A)$ we find $\psi\in E_{q_1,q_2}(A/m)$ with $w_2\circ \phi_0^{-1}\circ \phi$ homotopic to $\psi\circ w_2$. For each $\psi\in \text{Map}(W_2)$ there are at most $m$ such $\phi$.
			\end{lemm}
			\begin{proof}
				Since
				$
				T_1\circ \phi_0\circ T_2=\phi_0
				$
				in $\text{Map}(S)$, the mapping class $\phi_0$ conjugates $\langle T_2\rangle$ onto $\langle T_1\rangle$, and therefore descends to a homeomorphism $\bar \phi_0:W_2\rightarrow W_1$. Set
				$$
				q_1=u_1\circ \bar \phi_0\quad\mbox{and}\quad q_2=u_2.
				$$
				Take $\phi\in E_{p_1,p_2}^{T_1,T_2}(A)$. Then $\phi$  descends to a homeomorphism $\bar \phi:W_2\rightarrow W_1$ and we set $\psi=\bar \phi_0^{-1}\circ \bar \phi.$  It follows that
				\begin{align}\label{homopties.lemma}
					p_1\circ \phi&= u_1\circ w_1\circ \phi= u_1\circ \bar \phi \circ w_2=q_1\circ \psi\circ w_2,\\ \notag
					p_2&=u_2\circ w_2=q_2\circ w_2.
				\end{align}
				Consider $(S,\iota)\in\Pi(p_1,p_2,\phi)$. From the proof of Lemma \ref{top.argument2} we see the existence of a diffeomorphism $T$ so that $\iota\circ T\simeq (p_1\circ\phi,p_2)$. Thus we obtain from \eqref{homopties.lemma} the existence of $(W_2,\hat \iota)\in \Pi(q_1,q_2,\psi)$ such that $\iota\circ T\simeq\hat\iota\circ w_2$. Likewise, if $(W_2,\hat \iota)\in \Pi(q_1,q_2,\psi)$, then  $(S,\hat\iota\circ w_2)\in\Pi(p_1,p_2,\phi)$.
				
				The map $w_2$ is a  cover of degree $m$ and so, from uniqueness of least-area immersions \cite{markovic}, 
				$$\text{area}_{\sigma\oplus\rho}(\Pi(p_1,p_2,\phi))
				=m\,\text{area}_{\sigma\oplus\rho}(\Pi(q_1,q_2,\psi)).
				$$
				Hence $\psi\in E_{q_1,q_2}(A/m)$.
				
				Finally, after choosing a representative of $\psi$, its lifts through the regular cover $w_2$ differ by deck transformations. Hence there are at most $m$ such lifts $\eta$, and therefore at most $m$ corresponding $\phi=\phi_0\circ \eta$.

			\end{proof}
			Arguing exactly as in  \eqref{growth.multiples}, but now on the genus $r$ surface $W_2$, we obtain
			\[
			\#E_{p_1,p_2}^{T_1,T_2}(A)\leq m \#E_{q_1,q_2}(A/m)=O(A^{6(r-1)})=o(A^{6(k-1)}).
			\]
			
			Dividing the Burnside formula by $A^{6(k-1)}$ and letting $A\to\infty$, we obtain
			\[
			\lim_{A\to\infty}\frac{\#N_{p_1,p_2}(A)}{A^{6(k-1)}}=\lim_{A\to\infty}\frac{\#E_{p_1,p_2}(A)+o(A^{6(k-1)})}{|D(p_1)||D(p_2)|A^{6(k-1)}}=\frac{B(p_1^*\sigma)B(p_2^*\rho)}{|D(p_1)||D(p_2)|b_k}.
			\]
			Summing over the disjoint union $\mathcal L^*(k)=\cup_{p_1,p_2\in C_g(k)}\mathcal L^*_{p_1,p_2}(k)$ we obtain the stated formula.
		\end{proof}
	\end{proof}
\subsection{Proof of Theorem \ref{rigidity.thm}}
	
	{To prove the following theorem we study the asymptotic behavior of homotopy classes corresponding to large powers of Dehn twists}  (see \cite{slegers}).
	\begin{thm}
		Consider pairs of nonpositively curved metrics $(\sigma,\rho)$ and $(\hat\sigma,\hat \rho)$  so that
		$$A(\sigma\oplus\rho)=A(\hat\sigma\oplus\hat\rho).$$
		Then there is $c_0>0$ so that  $\sigma$ and $c_0\rho$ have the same simple length spectrum as $c_0\hat \sigma$ and $\hat \rho$, respectively.
	\end{thm}
	\begin{proof}
		Recall that $\mathcal S$ denotes the conjugacy classes of simple closed curves. Consider a conjugacy class $[\gamma]\in\mathcal S$ and the Dehn twist $D_\gamma\in {\text{Map}(\Sigma)}$. 
	
	With $\phi_n=(D_\gamma^n)^{-1}\in \text{Map}(\Sigma)$, $n\in\N$, and $\psi\in \text{Diff}_+(\Sigma)$, we have from Lemma \ref{area.homotopic.class} that
		\begin{equation}\label{homotopie.invariance}
			A(\sigma\oplus\rho)(\Pi(\phi_n\psi^{-1}))=A(\sigma,(\phi_n\psi^{-1})_*\rho)=A(\sigma,(\phi_n)_*(\psi^*\rho)).
		\end{equation}
		Combining with Proposition \ref{prop2} and using $\lambda_{(\phi_n)_*\psi^*\rho}=\phi_n(\lambda_{\psi^*\rho})$ we see that 
		\begin{align*}
			\lim_{n\to\infty}\frac{A(\sigma\oplus\rho)(\Pi(\phi_n\psi^{-1}))}{n}&=\lim_{n\to\infty}\frac{A(\sigma,(\phi_n)_*(\psi^*\rho))}{n}\\
			&= \lim_{n\to\infty}\frac{i(\lambda_\sigma,\lambda_{(\phi_n)_*\psi^*\rho})}{n}=\lim_{n\to\infty}\frac{i(\phi_n^{-1}(\lambda_\sigma),\lambda_{\psi^*\rho})}{n}.
		\end{align*}
		From \eqref{dehn.twist.limit} we know that 
		$$
		\lim_{n\to\infty}\frac{\phi_n^{-1}(\lambda_{\sigma})}{n}=\lim_{n\to\infty}\frac{D_\gamma^n(\lambda_\sigma)}{n}= i(\lambda_\sigma,\delta_\gamma)\delta_\gamma
		=l_\sigma([\gamma])\delta_\gamma.
		$$
		Therefore
		$$
		\liminf_{n\to\infty}\frac{A(\sigma\oplus\rho)(\Pi(\phi_n\psi^{-1}))}{n}
		=l_\sigma([\gamma])\,l_{\psi^*\rho}([\gamma])=l_\sigma([\gamma])\,l_{\rho}([\psi(\gamma)]).
		$$
Therefore, the equality $A(\sigma\oplus\rho)=A(\hat\sigma\oplus\hat\rho)$   implies that 
		\begin{equation}\label{equality.curves}
			l_\sigma([\gamma])\,l_{\rho}([\psi(\gamma)])
			=
			l_{\hat \sigma}([\gamma])\,l_{\hat\rho}([\psi(\gamma)])  \quad\text{for all }[\gamma]\in \mathcal S, \psi\in\text{Map}(\Sigma).
		\end{equation}
		Consider the orbit 
		\begin{equation}\label{orbit}
			\mathcal O(\gamma)=\{t\psi(\delta_\gamma):t>0,\psi \in \text{Map}(\Sigma)\}\subset \mathcal{ML}(\Sigma). 
		\end{equation}
		The set $\mathcal O(\gamma)$ is dense in $\mathcal{ML}(\Sigma)$ by Theorem~6.19 in \cite{fathi-book}. Combining with the denseness of $\{t\delta_\gamma:t>0,\gamma \in \mathcal S\}$ in $\mathcal{ML}(\Sigma)$ we obtain from \eqref{equality.curves} that
		\begin{equation*}
			l_\sigma(\mu)\,l_{\rho}(\nu)
			=
			l_{\hat \sigma}(\mu)\,l_{\hat\rho}(\nu) \quad\text{for all }\mu,\nu \in\mathcal{ML}(\Sigma).
		\end{equation*}
		Choose $\nu_0\in\mathcal{ML}(\Sigma)\setminus\{0\}$ and set $c=l_{\hat\rho}(\nu_0)/l_\rho(\nu_0)$. Then
		\begin{equation*}
			l_\sigma(\mu)=c\,l_{\hat \sigma}(\mu)
			\quad\text{and}\quad
			l_{\rho}(\mu)=c^{-1}l_{\hat\rho}(\mu)
			\quad\text{for all }\mu\in \mathcal{ML}(\Sigma).
		\end{equation*}
		Setting $c_0=c^2$, we obtain the desired result.
	\end{proof}
	
\subsection{Proof of Corollary \ref{area.rigidity}.}
	
	\begin{cor}
		Consider K\"{a}hler--Einstein metrics $h=\sigma\oplus\tau$ and $\bar h=\bar\sigma\oplus \bar\tau$, not necessarily with the same Einstein constant, and so that
		$$
		A(h)=A(\bar h).
		$$
		Then $h$ and $\bar h$ are isometric and the isometry is homotopic to the identity.
	\end{cor}
	\begin{proof}
		Since $h$ and $\bar h$ are K\"{a}hler--Einstein, there exist $c,\bar c>0$ and hyperbolic metrics $\sigma,\tau,\bar\sigma,\bar\tau$ on $\Sigma$ such that
		$$
		h=c(\sigma\oplus\tau)\quad\text{and}\quad \bar h=\bar c(\bar\sigma\oplus\bar\tau).
		$$
		The hypothesis $A(h)=A(\bar h)$ can therefore be written as
		$$
		A(c\sigma\oplus c\tau)=A(\bar c\bar\sigma\oplus \bar c\bar\tau).
		$$
		Applying Theorem \ref{rigidity.thm}, we obtain $d_0>0$ such that $c\sigma$ and $d_0\bar c\bar\sigma$ have the same simple length spectrum, and $d_0c\tau$ and $\bar c\bar\tau$ have the same simple length spectrum. Since lengths scale by the square root of the scaling factor, this means
		$$
		l_\sigma([\gamma])=\sqrt{\frac{d_0\bar c}{c}}\,l_{\bar\sigma}([\gamma])\quad\text{and}\quad l_\tau([\gamma])=\sqrt{\frac{\bar c}{d_0c}}\,l_{\bar\tau}([\gamma])
		$$
		for every simple closed curve $[\gamma]$ on $\Sigma$.
		
		By Thurston's Compactification Theorem (see \cite[Proposition~7.12]{fathi-book}), hyperbolic metrics with proportional simple length spectra determine the same element in $\mathcal T(\Sigma)$. Hence the two proportionality constants above are equal to $1$. Therefore
		$$
		c=d_0\bar c\quad\text{and}\quad \bar c=d_0c.
		$$
		It follows that $d_0=1$ and $c=\bar c$. In particular, $\sigma$ and $\bar\sigma$ have the same simple length spectrum, and so do $\tau$ and $\bar\tau$. Applying Thurston's Compactification Theorem once more, we conclude that $\sigma$ and $\bar\sigma$ determine the same point of $\mathcal T(\Sigma)$, and likewise for $\tau$ and $\bar\tau$. Thus there are diffeomorphisms $f_1,f_2:\Sigma\rightarrow\Sigma$, both homotopic to the identity, such that
		$$
		f_1^*\bar\sigma=\sigma\quad\text{and}\quad f_2^*\bar\tau=\tau.
		$$
		Consequently,
		$$
		(f_1\times f_2)^*\bar h=(f_1\times f_2)^*(c\bar\sigma\oplus c\bar\tau)=c\sigma\oplus c\tau=h.
		$$
		Hence $f_1\times f_2$ is an isometry from $(\Sigma\times\Sigma,h)$ to $(\Sigma\times\Sigma,\bar h)$, and it is homotopic to the identity.
	\end{proof}
	
\subsection{Proof of Corollary \ref{rigidity.cor}.}
	
	\begin{cor}
		Consider $h,\hat h$ nonpositively curved metrics on $\Sigma\times \Sigma$ so that for every $k\geq g$ and $\Pi\in \mathcal S^*(k)$ we have
		$$
		\text{area}_{h}(\Pi)=\text{area}_{\hat h}(\Pi).
		$$
		Then $h$ and $\hat h$ are isometric and the isometry is homotopic to the identity.
	\end{cor}
	\begin{proof}
		By \cite[Theorem~2]{jost-yau} we obtain isometries homotopic to the identity from $(\Sigma\times\Sigma,h)$ and $(\Sigma\times\Sigma,\hat h)$ to product metrics
		$
		\sigma\oplus\tau\quad\text{and}\quad \hat\sigma\oplus\hat\tau,
		$
		where $\sigma,\tau,\hat\sigma,\hat\tau$ are nonpositively curved metrics on $\Sigma$. Replacing $h$ and $\hat h$ by these product metrics, we may therefore assume from now on that
		$$
		h=\sigma\oplus\tau\quad\text{and}\quad \hat h=\hat\sigma\oplus\hat\tau.
		$$
		
		Fix $x_0,y_0\in \Sigma$ and let $\Pi_1,\Pi_2\in \mathcal S^*(g)$ be the subgroup classes represented by $x\mapsto (x,y_0)$ and $x\mapsto (x_0,x)$. We claim that
		\begin{equation}\label{identical.areas}
			\text{area}_{\sigma\oplus\tau}(\Pi_1)=\text{area}_{\sigma}(\Sigma)\quad\text{and}\quad \text{area}_{\sigma\oplus\tau}(\Pi_2)=\text{area}_{\tau}(\Sigma).
		\end{equation}
		Indeed, the maps $x\mapsto (x,y_0)$ and $x\mapsto (x_0,x)$ give the upper bounds. For the lower bound, consider $(\Sigma,\iota)\in \Pi_1$ and write $\iota=(\iota_1,\iota_2)$. Since $(\iota_1)_*$ is an automorphism of $\pi_1(\Sigma)$, the map $\iota_1$ has degree $\pm 1$. Thus
		$$
		\text{area}_{\iota^*(\sigma\oplus\tau)}(\Sigma)\geq \left|\int_\Sigma \iota_1^*\omega_\sigma\right|=\text{area}_{\sigma}(\Sigma).
		$$
		The argument for $\Pi_2$ is identical. By the hypothesis of the corollary we obtain
		$$
		\text{area}_{\sigma}(\Sigma)=\text{area}_{\hat\sigma}(\Sigma)\quad\text{and}\quad \text{area}_{\tau}(\Sigma)=\text{area}_{\hat\tau}(\Sigma).
		$$
		
		We now claim that if $p:S\to \Sigma$ is a finite cover and a local diffeomorphism, then there is $c_0>0$, independent of $p$, so that
		$$
		l_{p^*\sigma}([\gamma])=\sqrt{c_0}\,l_{p^*\hat\sigma}([\gamma])\quad\text{and}\quad \sqrt{c_0}\,l_{p^*\tau}([\gamma])=l_{p^*\hat\tau}([\gamma])
		$$
		for every conjugacy class $[\gamma]$ of a simple closed curve in $S$.
		
		We first prove this when $S=\Sigma$. Since the hypothesis gives
		$$
		\text{area}_{\sigma\oplus\tau}(\Pi(\phi))=\text{area}_{\hat\sigma\oplus\hat\tau}(\Pi(\phi))
		$$
		for every $\phi\in \text{Map}(\Sigma)$, Theorem \ref{rigidity.thm} gives $c_0>0$ such that
		$$
		l_\sigma([\gamma])=\sqrt{c_0}\,l_{\hat\sigma}([\gamma])\quad\text{and}\quad \sqrt{c_0}\,l_\tau([\gamma])=l_{\hat\tau}([\gamma])
		$$
		for every simple closed curve $[\gamma]$ on $\Sigma$.
		
		Now let $p:S\to \Sigma$ be arbitrary. With the notation from the proof of Proposition \ref{higher.genus.count.lagrangian}, for every $\phi\in \text{Map}(S)$ we have
		$$
		A(p^*\sigma\oplus p^*\tau)(\hat\Pi(\phi))=\text{area}_{\sigma\oplus\tau}(\Pi(p,p,\phi))
		$$
		and likewise
		$$
		A(p^*\hat\sigma\oplus p^*\hat\tau)(\hat\Pi(\phi))=\text{area}_{\hat\sigma\oplus\hat\tau}(\Pi(p,p,\phi)).
		$$
		From Theorem \ref{rigidity.thm} (with $\Sigma$ replaced by $S$) we obtain $c(p)>0$ such that
		$$
		l_{p^*\sigma}([\gamma])=\sqrt{c(p)}\,l_{p^*\hat\sigma}([\gamma])\quad\text{and}\quad \sqrt{c(p)}\,l_{p^*\tau}([\gamma])=l_{p^*\hat\tau}([\gamma])
		$$
		for every simple closed curve $[\gamma]$ on $S$. To see that $c(p)=c_0$, let $[\alpha]$ be a simple closed curve on $\Sigma$ and let $\tilde\alpha$ be any component of $p^{-1}(\alpha)$. Then $\tilde\alpha$ is simple and for some $d\in \N$ we have
		$$
		l_{p^*\sigma}([\tilde\alpha])=d\,l_\sigma([\alpha])\quad\text{and}\quad l_{p^*\hat\sigma}([\tilde\alpha])=d\,l_{\hat\sigma}([\alpha]).
		$$
		Comparing the equalities above gives $c(p)=c_0$, proving the claim.
		
		Let now $[\gamma]$ be an arbitrary conjugacy class of $\pi_1(\Sigma)$. By \cite{scott}, there is a finite cover $p:S\to \Sigma$ and a simple closed curve $\tilde\gamma$ in $S$ whose image represents a positive power of $[\gamma]$. Thus for some $m\in \N$ we have
		$$
		l_{p^*\sigma}([\tilde\gamma])=m\,l_\sigma([\gamma])\quad\text{and}\quad l_{p^*\hat\sigma}([\tilde\gamma])=m\,l_{\hat\sigma}([\gamma]),
		$$
		and likewise for $\tau$ and $\hat\tau$. From the previous paragraph we conclude that
		$$
		l_\sigma([\gamma])=\sqrt{c_0}\,l_{\hat\sigma}([\gamma])\quad\text{and}\quad \sqrt{c_0}\,l_\tau([\gamma])=l_{\hat\tau}([\gamma])
		$$
		for every conjugacy class $[\gamma]$ of $\pi_1(\Sigma)$.
		
		By Croke--Fathi--Feldman \cite{croke2}, $(\Sigma,\sigma)$ is isometric to $(\Sigma,c_0\hat\sigma)$ and $(\Sigma,\tau)$ is isometric to $(\Sigma,c_0^{-1}\hat\tau)$, in both cases by an isometry homotopic to the identity. From \eqref{identical.areas}
		we get $c_0=1$. Therefore $\sigma$ is isometric to $\hat\sigma$ and $\tau$ is isometric to $\hat\tau$, both through isometries homotopic to the identity.
	\end{proof}
	\appendix
\section{Finiteness of area spectrum}\label{evidence.conjecture}
	
	Given an arbitrary basepoint $\sigma\in \mathcal T(\Sigma)$, consider
	\begin{equation}\label{PML}
		\mathcal{ML}(\sigma)=\{\tau \in \mathcal{ML}(\Sigma):i(\tau,\lambda_\sigma)=1\}.
	\end{equation}
	Using Proposition \ref{prop2} and Lemma \ref{area.homotopic.class}, one can see that for every divergent pair $\sigma_n,\rho_n\in \mathcal T(\Sigma)$,  there are $\varepsilon_n\to 0$ and $\mu, \nu \in \mathcal{ML}(\sigma)$ so that, after passing to a subsequence, 
	$$\lim_{n\to\infty}\varepsilon_n A(\sigma_n\oplus\rho_n)(\Pi(\phi))= i(\mu,\phi(\nu))\quad \mbox{for all }\phi\in  \text{Map}(\Sigma).$$
	\begin{thm}\label{finiteness.rigidity.thm} There is a finite set $M_0\subset \text{Map}(\Sigma)$ so that the map below is injective:
		$$\Phi:\mathcal{ML}(\sigma)\times \mathcal{ML}(\sigma)\rightarrow \R^{\# M_0},\quad \Phi(\mu,\nu)=\{i(\mu,\phi(\nu))\}_{\phi\in M_0}.$$
	\end{thm}
	\begin{proof} Consider a filtration $M_i\subset M_{i+1}\subset \cdots \subset\text{Map}(\Sigma)$ where $M_i^{-1}=M_i$, $\cup_{i\in\N}M_i=\text{Map}(\Sigma)$, and $\#M_i<+\infty$. We argue by contradiction and suppose the theorem does not hold. In this case we can find $\mu_i,\tilde \mu_i, \nu_i,\tilde\nu_i$ in $\mathcal{ML}(\sigma)$ so that 
		\begin{equation}\label{hypothesis}
			i(\mu_i,\phi(\nu_i))=i(\tilde\mu_i,\phi(\tilde \nu_i))\quad\text{for all }\phi\in M_i
		\end{equation}
		but $(\mu_i,\nu_i)\neq (\tilde\mu_i,\tilde\nu_i)$. The space $\mathcal{ML}(\sigma)$  is compact and so, after passing to a subsequence, we assume that $(\mu_i,\nu_i)\to (\mu,\nu)$ and $(\tilde\mu_i,\tilde\nu_i)\to (\tilde \mu,\tilde\nu)$, where $\mu,\tilde \mu, \nu,\tilde\nu\in \mathcal{ML}(\sigma)$.
		\begin{lemm}\label{limits.equal} We have $\mu=\tilde \mu$ and $\nu=\tilde\nu$.
		\end{lemm}
		\begin{proof}
			From continuity of intersection we have that $i(\mu,\phi(\nu))=i(\tilde\mu,\phi(\tilde \nu))$ for all $\phi \in \text{Map}(\Sigma)$. Consider a simple closed geodesic $\gamma$ and the Dehn twist $D_\gamma\in  \text{Map}(\Sigma)$.  We obtain from \eqref{dehn.twist.limit} that for all $\phi \in \text{Map}(\Sigma)$
			\begin{align*}
				i(\mu,\phi(\delta_\gamma))i(\nu,\delta_\gamma)&=\lim_{n\to\infty}\frac{i(\mu,\phi \circ D_\gamma^n(\nu))}{n}=\lim_{n\to\infty}\frac{i(\tilde \mu,\phi \circ D_\gamma^n(\tilde\nu))}{n}\\
				&=i(\tilde \mu,\phi(\delta_\gamma))i(\tilde \nu,\delta_\gamma).
			\end{align*}
			Using the denseness of $\mathcal O(\gamma)$ (see \eqref{orbit})  in $\mathcal{ML}(\Sigma)$ we conclude that
			$$i(\mu,\tau)i(\nu,\delta_\gamma)= i(\tilde\mu,\tau)i(\tilde\nu,\delta_\gamma) \quad\text{for all }\tau \in\mathcal{ML}(\Sigma).$$
			Using the denseness of  $\{t\delta_\gamma:t>0,\gamma \in \mathcal S\}$ in $\mathcal{ML}(\Sigma)$  we conclude that
			$$i(\mu,\tau)i(\nu,\lambda)= i(\tilde\mu,\tau)i(\tilde\nu,\lambda) \quad\text{for all }\lambda,\tau \in\mathcal{ML}(\Sigma).$$ Choose $\lambda_0\in \mathcal{ML}(\Sigma)$ with $i(\nu,\lambda_0)>0$. Then $i(\tilde\nu,\lambda_0)>0$ as well, and
			$$i(\mu,\tau)=c\, i(\tilde\mu,\tau)\quad\text{for all }\tau \in\mathcal{ML}(\Sigma),\qquad c=\frac{i(\tilde\nu,\lambda_0)}{i(\nu,\lambda_0)}.$$
			Hence $\mu=c\tilde\mu$. Exchanging the roles of $(\mu,\tilde\mu)$ and $(\nu,\tilde\nu)$ we also get $\nu=c^{-1}\tilde\nu$. The fact that $\mu,\tilde \mu, \nu,\tilde\nu\in \mathcal{ML}(\sigma)$ implies that $c=1$.
		\end{proof}
		Consider convex open cones $V,V^*\subset  \R^{6(g-1)}$ described in Section \ref{currents.section} and the homeomorphisms onto their images $$T:V \rightarrow \mathcal{ML}(\Sigma), \quad T^*:V^*\rightarrow \mathcal{ML}(\Sigma).$$ Set $\Omega=T(V)$ and $\Omega^*=T^*(V^*)$. We argue that, without loss of generality, we can assume $\mu \in \Omega$ and $\nu\in \Omega^*$. Choose simple closed geodesics $\gamma, \gamma^*$ so that $\delta_\gamma\in \Omega$, $\delta_{\gamma^*} \in \Omega^*$, and $i(\mu,\delta_\gamma), i(\nu,\delta_{\gamma^*})>0$. From \eqref{dehn.twist.limit} and scale invariance of $\Omega, \Omega^*$ we see that we can choose $\bar n\in \N$ so that, setting $\psi=D^{\bar n}_{\gamma}$ and $\psi'=D^{\bar n}_{\gamma^*}$,
		$$\psi(\mu)=D^{\bar n}_{\gamma}(\mu)\in \Omega\quad\text{and}\quad \psi'(\nu)=D^{\bar n}_{\gamma^*}(\nu)\in \Omega^*.$$
		Set $u=\psi(\mu)$, $v=\psi'(\nu)$, and denote $\psi(\mu_i), \psi'(\nu_i), \psi(\tilde\mu_i), \psi'(\tilde\nu_i)$ by $u_i,v_i,\tilde u_i,\tilde v_i$, respectively. We have $(u_i,v_i), (\tilde u_i,\tilde v_i)\to (u,v)$ by Lemma \ref{limits.equal} and so we assume that
		\begin{equation}\label{sequence.face}
			u, u_i,\tilde u_i \in \Omega\quad\text{and}\quad  v, v_i, \tilde v_i \in \Omega^*\quad\text{for all }i\in\N.
		\end{equation}
		Note that $u\in  \mathcal{ML}(\psi_*(\sigma))$ because
		$$i(u,\lambda_{\psi_*\sigma})=i(u,\psi(\lambda_\sigma))=i(\psi(\mu),\psi(\lambda_\sigma))=i(\mu,\lambda_\sigma)=1.$$
		The same computation shows 
		$$ u, u_i, \tilde u_i\in \mathcal{ML}(\psi_*(\sigma)).$$  Furthermore we obtain from \eqref{hypothesis} that  given $\phi \in\text{Map}(\Sigma)$ we have for all $i\in\N$ large enough so that $\psi^{-1}\circ \phi\circ\psi'\in M_i$,
		\begin{align}\label{sequence.identity}
			i(u_i,\phi(v_i))&=i(\psi(\mu_i),\phi\circ\psi'(\nu_i))=i(\mu_i,\psi^{-1}\circ \phi\circ\psi'(\nu_i))\\ \notag
			&=i(\tilde \mu_i,\psi^{-1}\circ \phi\circ\psi'(\tilde \nu_i))=i(\tilde u_i,\phi(\tilde v_i)).
		\end{align}
		From \eqref{sequence.face} we find $x_i,\tilde x_i, \bar x\in V$ and  $y_i,\tilde y_i,\bar y \in V^*$ so that 
		$$T(\bar x)=u, \,T(x_i)=u_i, \,T(\tilde x_i)=\tilde u_i, \, T^*(\bar y)=v, \,T^*(y_i)=v_i, \,T^*(\tilde y_i)=\tilde v_i.$$
		We must have $x_i, \tilde x_i\to \bar x$ in $V$ and $y_i, \tilde y_i\to \bar y$ in $V^*$ because $T$ and $T^*$ are homeomorphisms onto their images. After passing to a further subsequence, one of the inequalities
		$$|x_i-\tilde x_i| \geq|y_i-\tilde y_i| \quad\text{or}\quad |y_i-\tilde y_i|\geq|x_i-\tilde x_i|$$
		holds for all $i\in\N$. Because $M_i=M_i^{-1}$, the argument is symmetric in the two factors, and so we assume that for all $i\in\N$
		$$r_i=|x_i-\tilde x_i| \geq|y_i-\tilde y_i|.$$
		Since $(u_i,v_i)\neq(\tilde u_i,\tilde v_i)$, we have $r_i>0$ for all $i\in\N$. After passing to a subsequence, we assume that
		$$\frac{x_i-\tilde x_i}{r_i}\to X\neq 0\quad\text{and}\quad \frac{y_i-\tilde y_i}{r_i}\to Y.$$
		
		Next we show that $X=c\bar x$ for some $c\neq 0$.
		\begin{lemm}\label{first.linear} If $\phi\in \text{Map}(\Sigma)$ is such that $\phi^{-1}(u)\in \Omega$ and $\phi(v)\in \Omega^*$, then setting $x(\phi^{-1})=T^{-1}(\phi^{-1}(u))\in V$ and $y(\phi)=(T^*)^{-1}(\phi(v))$, we have 
			$$X\cdot y(\phi)+Y \cdot x(\phi^{-1})=0.$$
		\end{lemm}
		
		\begin{proof}
			For all $i$ sufficiently large we find $x_i(\phi^{-1})\in V$, $y_i(\phi)\in V^*$ so that $\phi^{-1}(\tilde u_i)=T(x_i(\phi^{-1}))$ and $\phi(v_i)=T^*(y_i(\phi))$. We must  have $x_i(\phi^{-1})\to x(\phi^{-1})$ and $y_i(\phi)\to y(\phi)$.
			Using \eqref{sequence.identity} we see that for all $i\in\N$ large
			\begin{align*}
				0&=i(u_i,\phi(v_i))-i(\tilde u_i,\phi(\tilde v_i))\\
				&=i(u_i,\phi(v_i))-i(\tilde u_i,\phi(v_i))+i(\tilde u_i,\phi(v_i))-i(\tilde u_i,\phi(\tilde v_i))
			\end{align*}
			and so, recalling that $i(T(x),T^*(y))=x \cdot y$ (see Section \ref{currents.section}), we have
			\begin{align*}
				0&=i(u_i,\phi(v_i))-i(\tilde u_i,\phi(v_i))+i(\phi^{-1}(\tilde u_i),v_i)-i(\phi^{-1}(\tilde u_i),\tilde v_i)\\
				&=x_i\cdot y_i(\phi)-\tilde x_i\cdot y_i(\phi)+x_i(\phi^{-1})\cdot y_i-x_i(\phi^{-1})\cdot \tilde y_i.
			\end{align*}
			Therefore, dividing by $r_i$ on both sides and using that $x_i(\phi^{-1})$ and $y_i(\phi)$ converge to $x(\phi^{-1})$ and $y(\phi)$, respectively, we see that
			$$0=\lim_{i\to\infty}(r_i^{-1}(x_i-\tilde x_i)\cdot y_i(\phi)+r_i^{-1}x_i(\phi^{-1})\cdot (y_i-\tilde y_i))=X \cdot y(\phi)+Y \cdot x(\phi^{-1}).$$
		\end{proof}
		\begin{lemm}\label{second.linear}
			We have for all $x\in V$ and $y\in V^*$ that
			$$(x \cdot \bar y)X \cdot y+(y \cdot \bar x)Y  \cdot x=0.$$
		\end{lemm}
		\begin{proof}
			Consider $x\in V$ and $y\in V^*$ for which we can find $c>1$ and a pseudo-Anosov diffeomorphism $\phi\in  \text{Map}(\Sigma)$ so that 
			$$\phi(T(x))=c^{-1}T(x)\quad\mbox{and}\quad \phi(T^*(y))=cT^*(y).$$
			From Section \ref{mpg.subsection} we know that such pairs $(T(x),T^*(y))$ are dense in $\Omega\times\Omega^*$. 
			
			We also assume that neither $x$ nor $y$ is a multiple of $\bar x$ or $\bar y$. Thus, we have from \eqref{pseudo.anosov} and $i(T(x'),T^*(y'))=x'\cdot y'$ if  $x'\in V$ and $y'\in V^*$, that
			\begin{align}\label{pseudo.anosov2}
				\lim_{n\to\infty}\frac{\phi^n(v)}{c^n}&=\frac{i(T(x),v)}{i(T(x),T^*(y))}T^*(y)=\frac{x\cdot \bar y}{x\cdot y}T^*(y)\\ \notag
				\lim_{n\to\infty}\frac{\phi^{-n}(u)}{c^n}&=\frac{i(T^*(y),u)}{i(T(x),T^*(y))}T(x)=\frac{y\cdot \bar x}{x\cdot y}T(x).
			\end{align}
			Hence $\phi^n(v)\in \Omega^*$ and $\phi^{-n}(u)\in \Omega$ for all $n$ very large. Consider $x(\phi^{-n})\in V$, $y(\phi^n)\in V^*$ so that $T(x(\phi^{-n}))=\phi^{-n}(u)$ and $T^*(y(\phi^n))=\phi^{n}(v)$. Therefore we obtain from \eqref{pseudo.anosov2} that
			\begin{equation*}
				\lim_{n\to\infty}\frac{y(\phi^n)}{c^n}=\frac{x\cdot \bar y}{x\cdot y}y,\,\,\, \lim_{n\to\infty}\frac{x(\phi^{-n})}{c^n}=\frac{y\cdot \bar x}{x\cdot y}x.
			\end{equation*}
			Inserting in Lemma \ref{first.linear}, applied to $\phi^n$, and making $n\to\infty$ we have 
			$$X\cdot \frac{y(\phi^n)}{c^n}+Y\cdot \frac{x(\phi^{-n})}{c^n}=0\implies  (x\cdot \bar y)X \cdot y+(y \cdot \bar x)Y\cdot x=0.$$
			The fact that the pairs $(x,y)$ above are dense in $V\times V^*$ implies the desired result.
		\end{proof}
		
		Fix $x\in V$. Since $\bar y\in V^*$, we have $x\cdot \bar y\neq 0$. Considering the identity given by Lemma \ref{second.linear} for all $y$ in the open cone $V^*$ we deduce that $X=c\bar x$ for some $c\neq 0$. 
		
		We now argue that $c=0$, giving us a contradiction. Bonahon \cite[Section 6]{bonahon4} showed that  
		\begin{equation}\label{f.differentiable}
			f:V\rightarrow \R,\quad f(x)=l_{\psi_*\sigma}(T(x))
		\end{equation} is locally Lipschitz.
		We have $f(x_i)=f(\tilde x_i)=1$ and $\tilde x_i\to \bar x$. Thus $x_i=(1+cr_i)\tilde x_i+o(r_i)$ and so, from homogeneity, 
		$$1=f(x_i)=f((1+cr_i)\tilde x_i) +o(r_i)=(1+cr_i) f(\bar x_i)+o(r_i)=1+cr_i+o(r_i).$$ 
		Hence $c=0$.
		 
	\end{proof} 
	
	\bibliographystyle{amsbook}

\end{document}